\documentclass[11]{amsart}  
\usepackage{graphicx} 

\usepackage{amsbsy,mathtools,mathrsfs,amsmath,amsfonts,amssymb,amsthm}

\newcommand{\N}{\mathbb{N}}
\newcommand{\Z}{\mathbb{Z}}
\newcommand{\R}{\mathbb{R}}
\newcommand{\C}{\mathbb{C}}

\newcommand{\re}{\mathop{\mathrm{Re}}}
\newcommand{\im}{\mathop{\mathrm{Im}}}

\newcommand{\eps}{\varepsilon} 

\newcommand{\be}{\begin{equation}}
\newcommand{\ee}{\end{equation}}

\newcommand{\balpha}{{\boldsymbol{\alpha}}}
\newcommand{\cX}{\mathscr{X}} 
\newcommand{\cF}{\mathscr{F}} 

\newcommand{\ml}{\vskip 4pt\noindent}

\newcommand{\mr}{\mathbb{R}}
\newcommand{\mn}{\mathbb{N}}

\newcommand{\ba}{\boldsymbol{a}}

\newcommand{\fB}{\mathfrak{B}}
\newcommand{\fE}{\mathfrak{E}}

\numberwithin{equation}{section}

\DeclareMathOperator\gr{graph}

\DeclareMathOperator\supp{supp}

\newtheorem{remark}{Remark}
\newtheorem{proposition}{Proposition}

\newtheorem{theorem}{Theorem}
\newtheorem{definition}{Definition}

\begin{document}

\title{On Some Generalizations of B-Splines}

\author{Peter Massopust}

\address{Centre of Mathematics, Technical University of Munich\\ Boltzmannstr. 3,
85478 Garching b. M\"unchen, Germany\\
massopust@ma.tum.de}



\maketitle

\begin{abstract}
In this article, we consider some generalizations of polynomial and exponential B-splines. Firstly, the extension from integral to complex orders is reviewed and presented. The second generalization involves the construction of uncountable families of self-referential or fractal functions from polynomial and exponential B-splines of integral and complex orders. As the support of the latter B-splines is the set $[0,\infty)$, the known fractal interpolation techniques are extended in order to include this setting.
\end{abstract}

\section{Introduction}

Schoenberg's polynomial B-splines \cite{schoen} are a powerful tool in approximation theory because of their favorable analytic and computational properties. Unfortunately, polynomial B-splines also have some disadvantages. Amongst them, we list:
\begin{itemize}
\item Polynomial B-splines have only integer smoothness which is linked to the integer order $n$. However, for approximation-theoretic purposes, it is useful to fill in the gaps in the smoothness spectrum $C^n$, $n\in \N$. There are many functions that are elements of, for instance,  H\"older spaces $C^{n, \alpha}$, $0\leq\alpha < 1$.
\item Polynomial B-splines do not contain phase information. The importance of approximation functions to be able to provide phase information is exemplified by the so-called Oppenheim-Lim Experiment \cite{ol}. In their paper, Oppenheim \& Lim showed that the Fourier reconstruction of an image using only the modulus of the complex-valued Fourier coefficients results in less informative content than a reconstruction from the phase of the Fourier coefficients (and setting the modulus equal to 1). The reconstruction from phase showed singularities such as corners and edges quite clearly but they were hard to see in the reconstruction from the modulus. 

In addition, there are sometimes requirements for a single-band frequency analysis. For some applications, e.g., for phase retrieval tasks, complex-valued analysis bases are needed since real-valued bases can only provide a symmetric spectrum.
\item Polynomial splines are ill-suited for describing functions or data that exhibit sudden growth or decay because of their oscillatory behavior near the points where such an increase or decrease occurs \cite{sb}. 
\end{itemize}

The first two items in the above list can be resolved by extending the order of B-splines from integral $n$ to complex $z$ with $\re z > 1$. The thus obtained so-called complex B-splines \cite{fub} generate a two-parameter family of functions with a continuous smoothness spectrum and built-in phase information. 

The third issue can be rectified by introducing exponential splines and B-splines. These splines are employed to model phenomena that follow differential  systems of the form $\dot{x} = A x$, where $A$ is a constant matrix. For such equations the solutions are linear combinations of functions of the type $e^{a x}$ and $x^n e^{a x}$, $a\in \R$. Like  polynomial B-splines, exponential B-splines can be defined as finite convolution products of exponential functions. See \cite{ammar,dm1,mccartin,sakai1, spaeth,unserblu05,zoppou} for an incomplete list of references for exponential splines. The extension of exponential B-splines to complex order \cite{m} adds the option of applying them for the retrieval of phase information. 

Neither the original nor extended polynomial and exponential B-splines are appropriate approximants when functions exhibit complex intrinsic characteristics such as self-referential or fractal behavior. In these cases, one needs to resort to fractal interpolation and approximation techniques to describe them. The extension of polynomial B-splines to an uncountable family of self-referential or fractal functions indexed by a finite tuple of real numbers $\alpha_i\in (-1,1)$ was presented in, i.e., \cite{m05,m10,ns}. Here we consider the case of exponential B-splines of integral order and also the fractal generalization of polynomial and exponential B-splines of complex orders. The latter requires extending fractal interpolation techniques to unbounded domains.

The structure of this article is as follows. For the sake of presentation and completeness, we briefly introduce polynomial and exponential B-splines and their complex order extensions in Sections \ref{pbsplines}, respectively, \ref{expbsplines}. A brief introduction to self-referential functions is provided in Section \ref{selfref} and in the final Section \ref{selfrefbexp} uncountably many families of self-referential polynomial and exponential B-splines of complex orders are constructed.

\section{Polynomial B-Splines}\label{pbsplines}

In this section, we briefly review polynomial splines and their basis functions, polynomial B-splines. The interested reader may consult the large literature on splines for more details and further results. 

To this end, let $\cX = \{a= x_0 < x_1 < \cdots < x_k<x_{k+1}=b\}$ be a set of points, called knots, supported on the real line $\R$.
\begin{definition}
A \emph{spline of order $n$} on $[a,b]$ with knot set $\cX$ is a function $s:[a,b]\to\mathbb{R}$ such that
\begin{enumerate}
\item[(i)] On each subinterval $[x_{i-1},x_i)$, $s$ is a polynomial of order at most $n$ (degree at most $n-1$);\ml
\item[(ii)]   $s\in C^{n-2}[a,b]$.
\end{enumerate}
$s$ is called a \emph{cardinal spline} if the knot set is a contiguous subset of $\Z$.
\end{definition}
The set ${\mathcal S}_{\cX,n}$ of all spline functions $s$ of order $n$ over a knot set $\cX$ forms an $\R$-vector space of dimension $n+k$. A convenient and powerful basis of ${\mathcal S}_{\cX,n}$ is given by Schoenberg's cardinal polynomial B-Splines \cite{schoen}. They are recursively defined as follows. Denote by $\chi$ the characteristic function on $[0,1]$ and set
\begin{align}
B_{1} (x) &:= \chi(x),\nonumber\\
B_{n}(x) &:= (B_{n-1}\ast B_{1})(x) = \int_0^1 B_{n-1} (x - t) dt,\qquad 2 \leq n \in \N,\label{eq1.1}
\end{align}
where $\ast$ denotes the convolution between functions. An immediate consequence of this definition is that $\supp B_n = [0,n]$ and that $B_n\in C^{n-2}$, $n\in \N$, with $C^{-1}$ denoting the family of piecewise continuous functions. Some graphs of these cardinal polynomial B-splines are shown in Figure \ref{fig1}.
\begin{figure}[h!]
\begin{center}
\includegraphics[width=5cm,height=2.5cm]{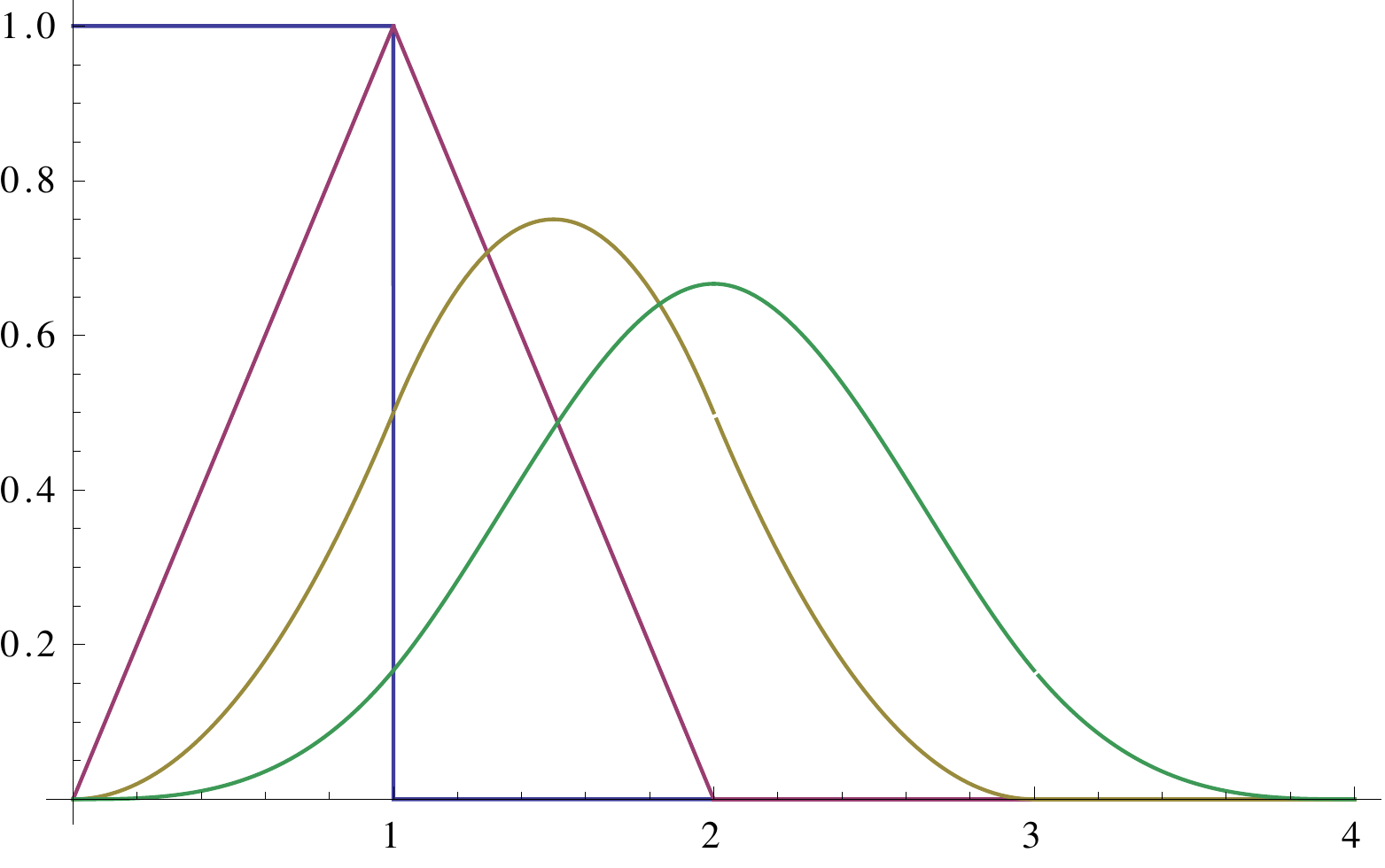}
\caption{Some graphs of polynomial B-splines: $n=1,2,3,4$.}\label{fig1}
\end{center}
\end{figure}

Taking the Fourier transform of \eqref{eq1.1} yields the Fourier representation of $B_n$, which is sometimes used to define the B-splines.
\be\label{eq1.2}
\widehat{B}_{n}(\omega) := \cF(B_n)(\omega)  := \int_\mr B_n(x) e^{- i\omega x} dx = \left(\frac{1-e^{- i\omega}}{i\omega} \right)^{n}.
\ee
It can be shown, either using \eqref{eq1.1} or \eqref{eq1.2} that the $n$-order B-spline has an explicit representation in the form
\be\label{eq2.3a}
B_n (x)  = \frac{1}{\Gamma (n)}\,\sum_{k=0}^{\infty} (-1)^k \binom{n}{k} (x-k)_+^{n-1},
\ee 
where $x_+ := \max\{0,x\}$.

The collection $\{B_n : n\in \mn\}$ is thus a {discrete} family of functions with increasing smoothness and support. Both the support and the smoothness are tied to the integral order $n$.
%

The next result justifies the term B-spline with B standing for basis. For a proof, see for instance \cite{deboor}.
\begin{proposition}
Every cardinal spline function $s:[a,b]\to$ of order $n$ has a unique representation in terms of a finite shifted sequence of cardinal B-splines of order $n$:
\[
s(x) = \sum\limits_{j=-n+1}^k c_j B_n (x-j),
\]
where $c_j\in \R$.
\end{proposition}
Hence, investigating properties of splines reduces to those of B-splines.
\ml
\textbf{Terminology.} As we are dealing exclusively with cardinal splines and B-splines in the remainder of this paper, we will drop the adjective ``cardinal.''
\subsection{Some properties of polynomial B-splines}\label{sect1.1}
The polynomial B-splines enjoy among others the following properties.
\begin{enumerate}
\item[(i)] \textit{Recursion Relation}:
\[
\forall n\in \N\;\forall x\in \R: B_n (x) = \frac{x}{n-1} B_{n-1} (x) + \frac{n-x}{n-1} B_{n-1}(x-1)
\]
\item[(ii)] \textit{Convolution Relation}:
\[
\forall m, n\in \N:\quad B_m \ast B_n = B_{m+n}
\]
\item[(iii)] \textit{Convergence to Gaussians}: As $n\to\infty$, $B_n$ converges  in $L^p$-norm, $2\leq p \leq \infty$, to a modulated Gaussian.
\item[(iv)] The \textit{Error of Approximation} for an $f\in C^n [a,b]$ on a uniform grid of mesh size $h$ by polynomial B-splines of order $n$ is $\mathcal{O}(h^n)$.
\end{enumerate}
The interested reader may consult the extensive literature on $B$-splines to learn about additional properties of this important family of functions in approximation theory.
%

\subsection{Polynomial B-splines of complex orders}

Both the first and second obstacle of polynomial B-splines mentioned in the introduction can be overcome by extending them to include complex orders (or complex degrees). This can be done in the Fourier domain as follows. (Cf. \cite{fub}.)
\begin{definition}
Suppose $z\in \C$ with $\re z > 1$. The B-spline of complex order $z$, for short complex B-spline, is given by $\widehat{B}: \R\to\C$, 
\be
\widehat{B}_z (\omega) := \left( \frac{1-e^{-i\omega}}{i\omega}\right)^z,
\ee
or more precisely,
\be\label{eq2.2}
\widehat{B}_z(\omega) = \underbrace{\widehat{B}_{\re z}(\omega)}_{\footnotesize{\textrm{continuous smoothness}}}\, \underbrace{e^{i \im z \ln \Omega (\omega)}}_{\footnotesize{\textrm{phase}}} \, \underbrace{e^{- \im z \arg \Omega(\omega)}}_{\footnotesize{\textrm{modulation}}},
\ee
where $\Omega(\omega) := \frac{1-e^{-i\omega}}{i\omega}$.
\end{definition}
\noindent
We remark that $\widehat{B}_z$ is well-defined as $\gr \Omega$ does not intersect the real axis.

The first factor in the product appearing in \eqref{eq2.2} is the Fourier transform of a so-called \emph{fractional B-spline} \cite{ub}. 
Some graphs of such B-splines of real order are depicted in Figure \ref{fig2}.
\begin{figure}[h!]
\begin{center}
\includegraphics[width=5cm,height=3cm]{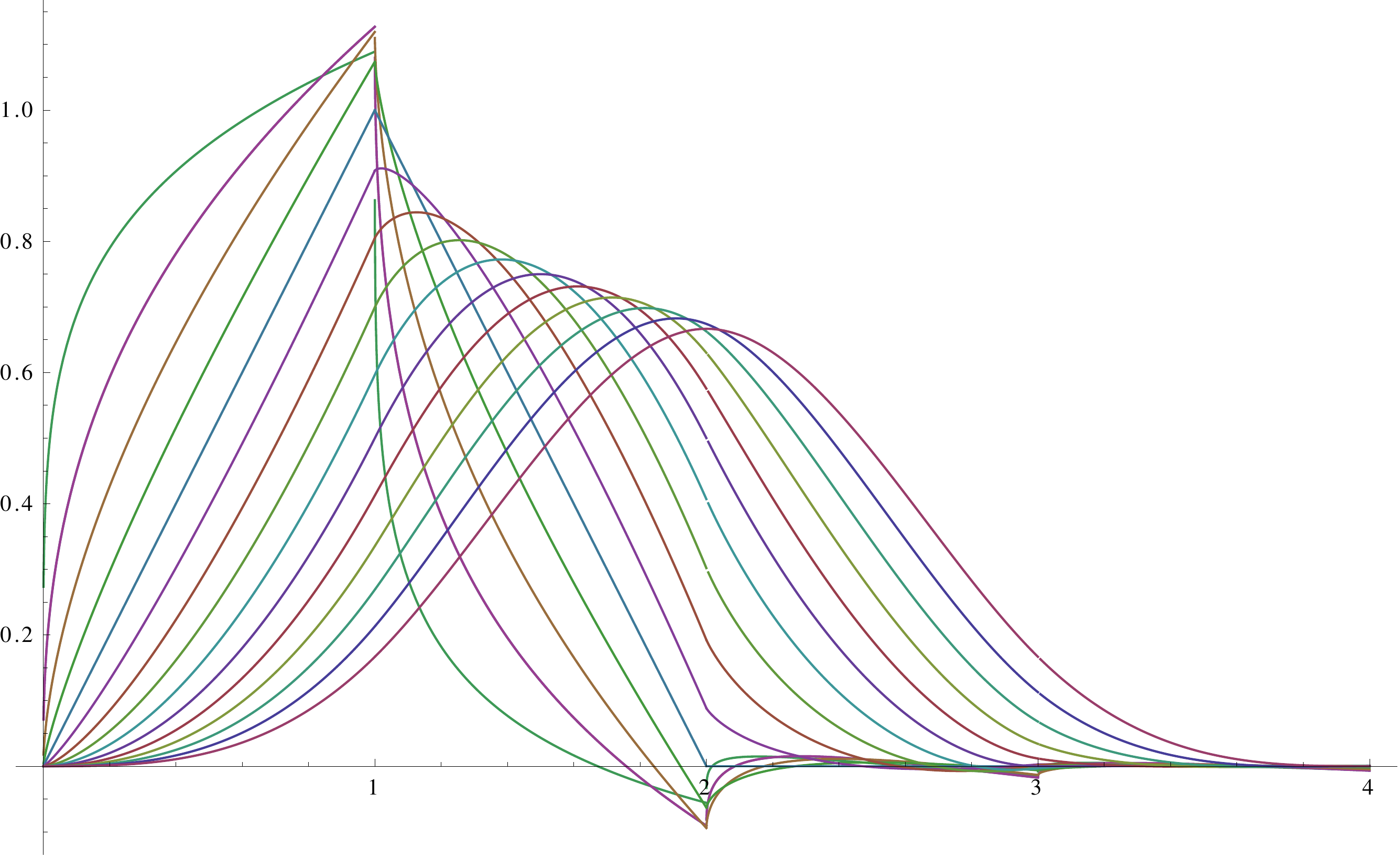}
\caption{A family of B-splines of real order for $\alpha = 0.6 + m\cdot 0.2$, $m = 1, \ldots, 17$.}\label{fig2}
\end{center}
\end{figure}
The second and third factors in \eqref{eq2.2} are a modulating and a damping factor. The presence of the imaginary part $\im z$ causes the frequency components on the negative and positive real axis to be enhanced with different signs. This has the effect of shifting the frequency spectrum towards the negative or positive frequency side, depending on the sign of $\im z$. The corresponding bases can be interpreted as approximate single-band filters \cite{fub}.

The time domain representation of a complex B-spline was derived in \cite{fub} and is given in the next theorem.

\begin{theorem}[Time domain representation]
\be\label{eq2.3}
B_z (x) = \frac{1}{\Gamma(z)} \sum_{k= 0}^\infty (-1)^k \left( {z} \atop {k}\right) (x-k)_+^{z-1},\quad \re z > 1.
\ee
Equality holds point-wise for all $x\in\R$ and in the $L^2(\R)$--norm.
\end{theorem}

Complex B-splines enjoy among others the following properties.
\begin{enumerate}
\item $B_z\in L^1 (\mathbb{R})\cap L^2 (\mathbb{R})$, $\re z > 1$.
\item $\displaystyle{\int_{\R}} B_z (x) dx = \widehat{B}_z (0) = 1$.
\item $B_z\in W^{p,2}(\mathbb{R})$ for $p < \re z - \frac{1}{2}$.
\item $B_z (x) = \mathcal{O}(|x|^{-m})$, for $m < \re z + 1$ and $|x|\to \infty$.
\item $B_z$ converges in $L^p$-norm, $2\leq p \leq \infty$, to a modulated and shifted Gaussian as $\re z \to \infty$.
\item $B_{\re z}$ reproduces polynomials up to order $\lceil \re z\rceil$.
\item For ${\re z} > 1$, $B_{\re z}$ is $({\re z}-1)$-H\"older continuous.
\item $\{B_z(\cdot - k)\}_{k\in \Z}$ is a Riesz sequence in $L^2(\R)$. This allows the construction of spline scaling functions and spline wavelets of complex order.
\end{enumerate}
\noindent
Some graphical examples of complex polynomial B-splines are shown in Figure \ref{fig3}.
\begin{figure}[h!]
\begin{center}
\includegraphics[width = 4cm, height = 2.5cm]{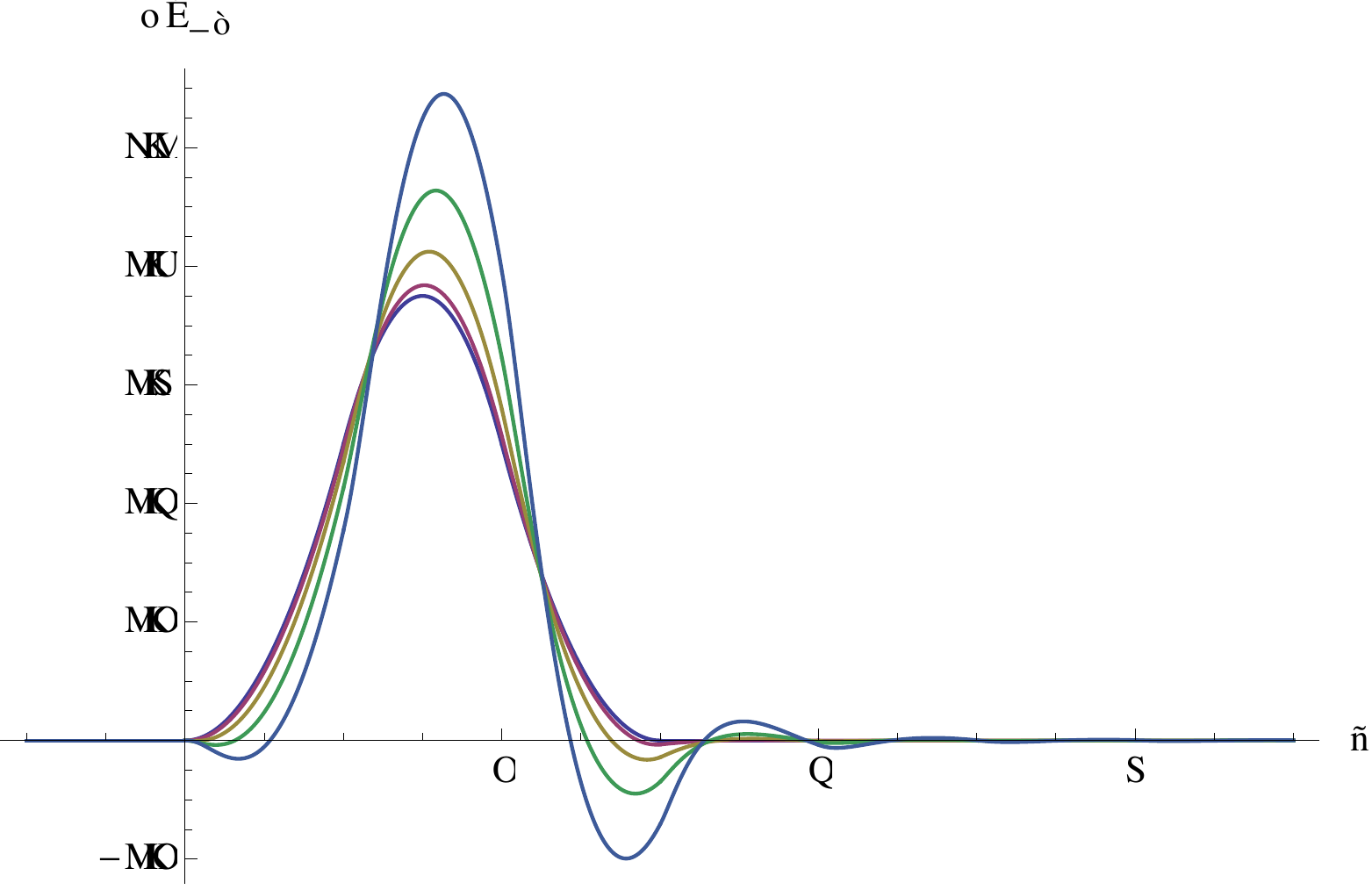}\hspace{1.5cm}
\includegraphics[width = 4cm, height = 2.5cm]{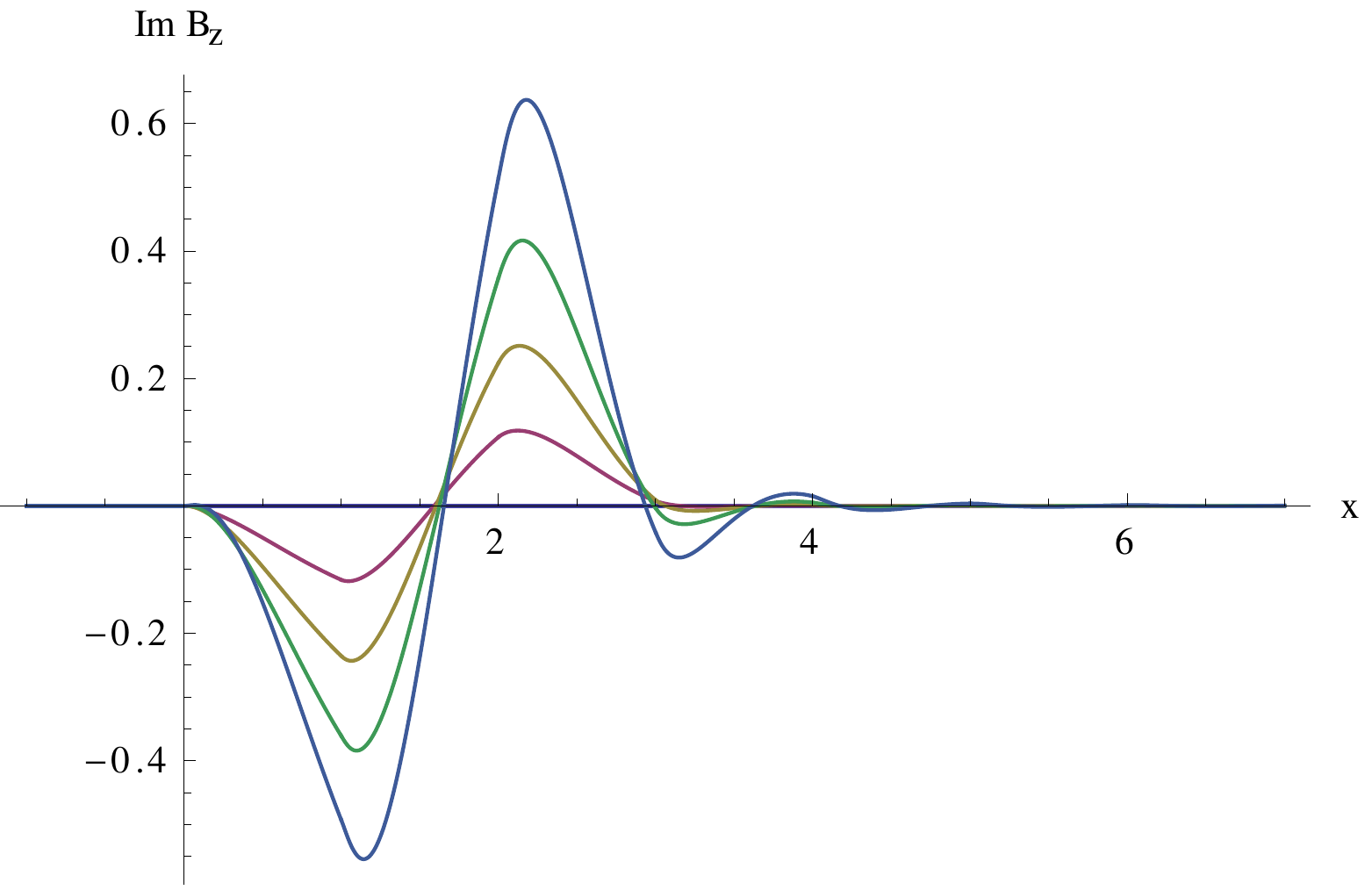}\\
\includegraphics[width = 5cm, height = 3cm]{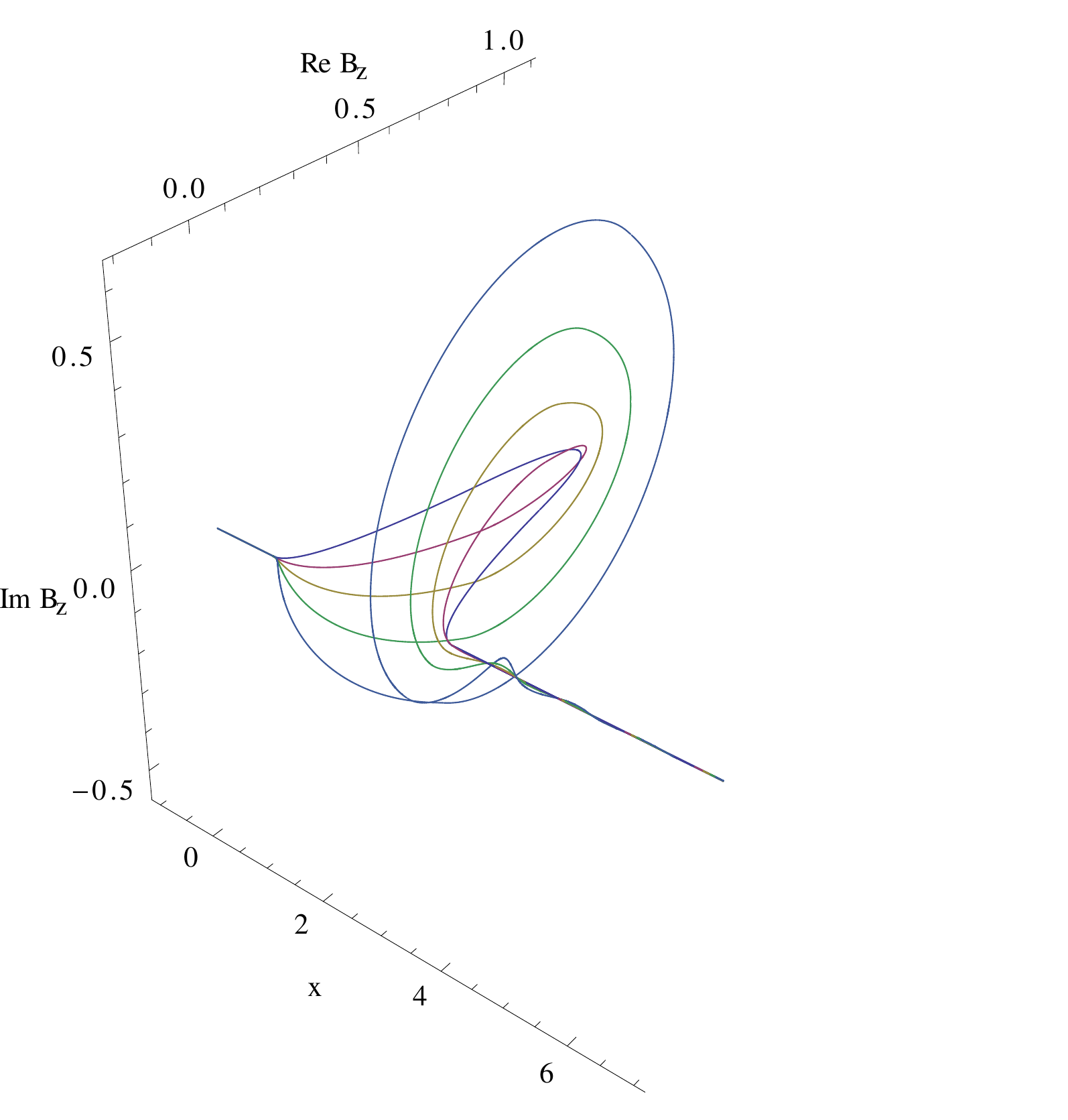}
\caption{$B_z$ for $z = (3 + \frac{k}{4}) + i$, $k = 0, 1, \ldots, 4$.}\label{fig3}
\end{center}
\end{figure}

In summary, complex B-splines are a continuous two-parameter family of functions which enjoy the properties:
\begin{enumerate}
\item[(a)] $\re z> 1$ gives a {continuous} family of functions of increasing smoothness $\re z$;
\item[(b)] $\im z$ contains phase information and can be used to describe and resolve singularities in signals and images.
\end{enumerate}

\section{Exponential B-Splines}\label{expbsplines}

Exponential B-splines can be used to interpolate or approximate data that exhibit sudden growth or decay and for which polynomial B-splines are not well-suited because of their oscillatory behavior near the points where the sudden growth or decay occurs \cite{sb}. The interested reader is referred to the following albeit incomplete list of references on exponential B-splines \cite{dm1,mccartin,spaeth,unserblu05,zoppou}.

To define the class of exponential B-splines, let $N\in \mn$ and let $\ba := (a_1, \ldots, a_N)$, where $a_1,\dots,
a_N\in \mr$ with $a_i\neq 0$ for at least one $i\in \N_N$.

\begin{definition}
An exponential B-spline $E_{N,\ba}:\mr\to \mr$ of order $N$ for the
$N$-tuple $\ba$ is a function of the form
\[
{E}_{N} := {E}_{N,\ba} :=  \underset{k = 1}{\overset{N}{*}} e^{a_k(\cdot)}\chi.
\]
\end{definition}
To simplify notation, we set $\eps^{a (\cdot)} := e^{a (\cdot)}\chi$. A closed formula for $E_n$ was derived in \cite{cm}. Note that $\supp E_N = [0,N]$, $N\in \N$.

%

%
For any $a\in \R$, the Fourier transform of $\eps^{-a(\cdot)}$ is given by 
\[
\mathcal{F}( \eps^{-a(\cdot)})(\omega) = \frac{1- e^{-a} e^{-i \omega}}{i \omega + a}.
\]
and, therefore,
\be\label{eq5.1}
\mathcal{F}({E}_n) (\omega) = \prod_{k=1}^n \frac{1- e^{-a_k} e^{-i \omega}}{i \omega + a_k}\;\;\overset{a_k = a}{=}\;\;\left(\frac{1- e^{-a} e^{-i \omega}}{i \omega + a}\right)^n.
\ee


%
\subsection{Exponential B-splines of complex order}

Let $z\in \C_{>1} := \{\zeta\in \C : \re \zeta > 1\}$ and $a > 0$. Taking the left-hand-side of \eqref{eq5.1} as a starting point, we define an exponential B-spline of complex order $z$, for short complex exponential B-spline, by (see \cite{m})
\begin{align}\label{eq5.2}
\widehat{E}_{z,a} (\omega) &:=\left(\frac{1-e^{-(a+i\omega)}}{a+i\omega}\right)^z\nonumber \\
& = \widehat{E}_{\re z, a} (\omega)\, e^{i\,\Omega_a (\omega)\, \im z}\, e^{- \arg\Omega_a (\omega)\, \im z },
\end{align}
where $\Omega_a(\omega) := \frac{1-e^{-(a+i\omega)}}{a+i\omega}$. An investigation of the function $\Omega_a:\R\to\C$ shows that 
$\widehat{E}_{z,a}$ is well-defined only if $a > 0$. (See \cite{m}.) The second and third terms in the product of \eqref{eq5.2} play the same role as they did in the case of complex polynomial B-splines. 

Using properties of the exponential difference operator and the definition of $E_{z,a}$, the following time domain representation of $E_{z,a}$ was proved in \cite{m}.
\begin{theorem}
Suppose $z\in \C_{>1}$ and $a > 0$. Then,
\[
E_{z,a} (x) = \frac{1}{\Gamma(z)}\,\sum_{k=0}^\infty (-1)^k\,\binom{z}{k}\, e^{-k a} e_+^{-a(x-k)}\,(x-k)_+^{z-1},
\]
where $e_+^{(\cdot)} := \chi_{[0,\infty)}\,e^{(\cdot)}$. The sum converges both point-wise in $\R$ and in the $L^2$--sense.
\end{theorem}

The figures below depicts some graphs of exponential B-splines of complex order.
\begin{figure}[h!]
\begin{center}
\includegraphics[width = 4cm, height = 2.5cm]{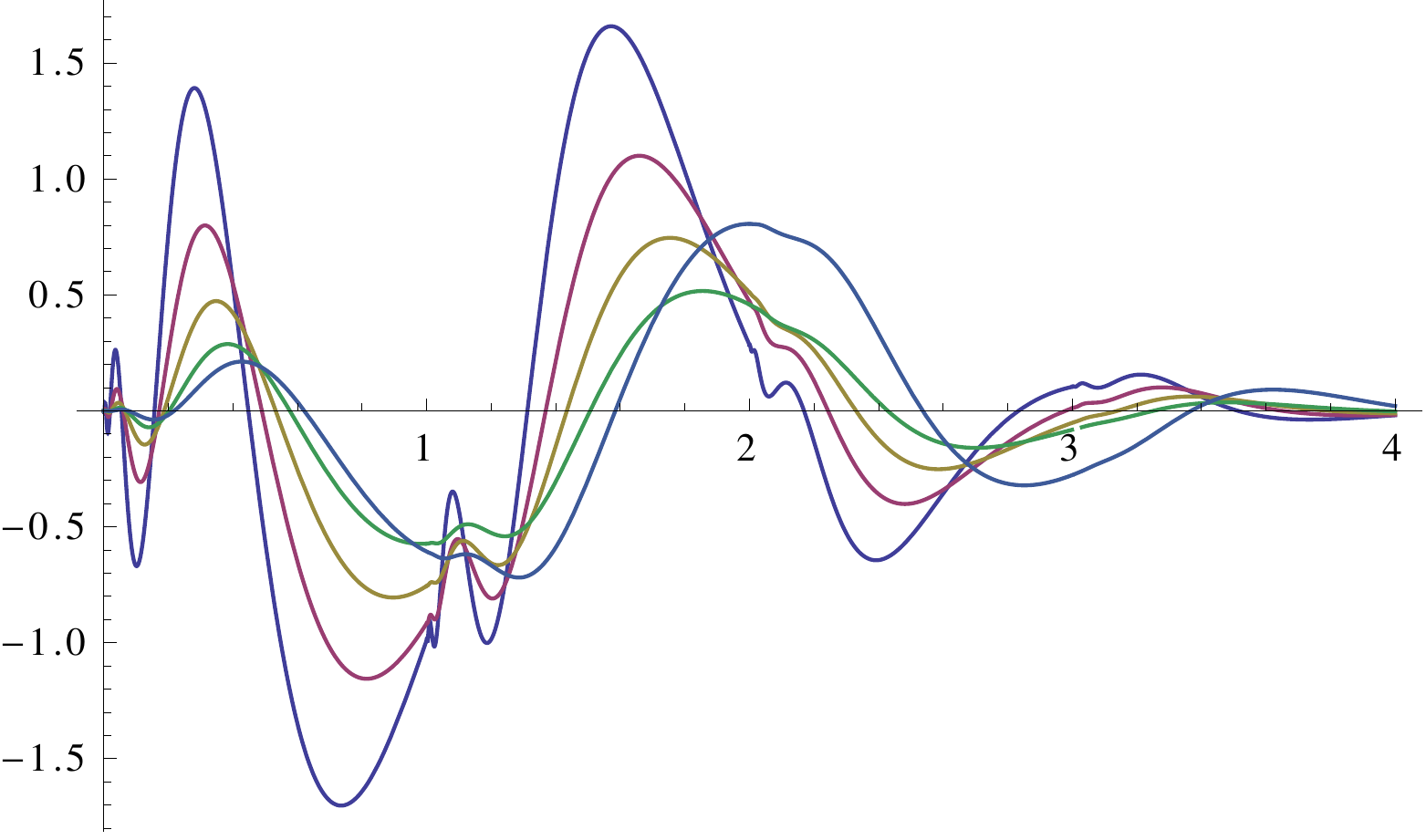}\hspace{1.5cm}
\includegraphics[width = 4cm, height = 2.5cm]{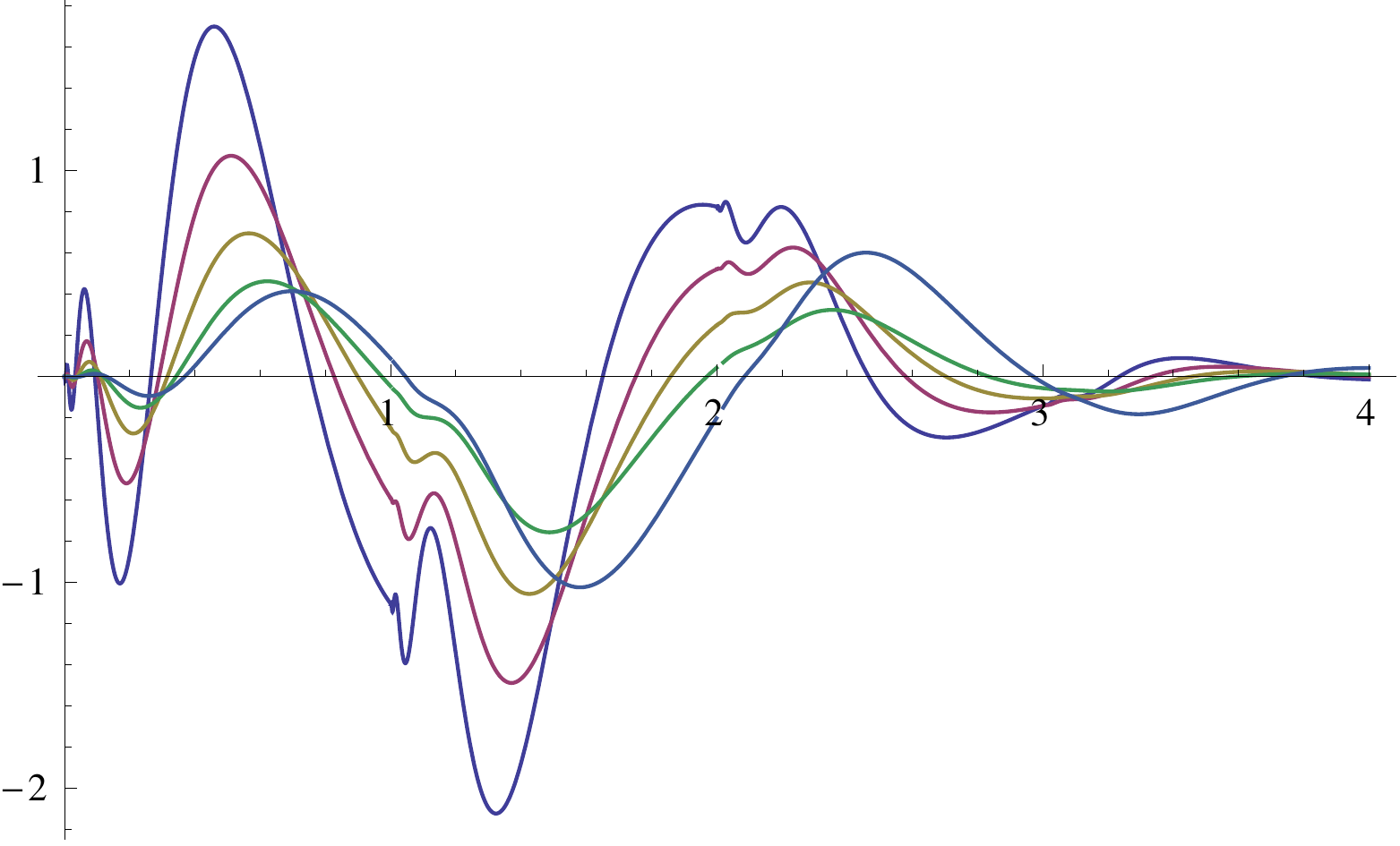}\\
\includegraphics[width = 5cm, height = 3cm]{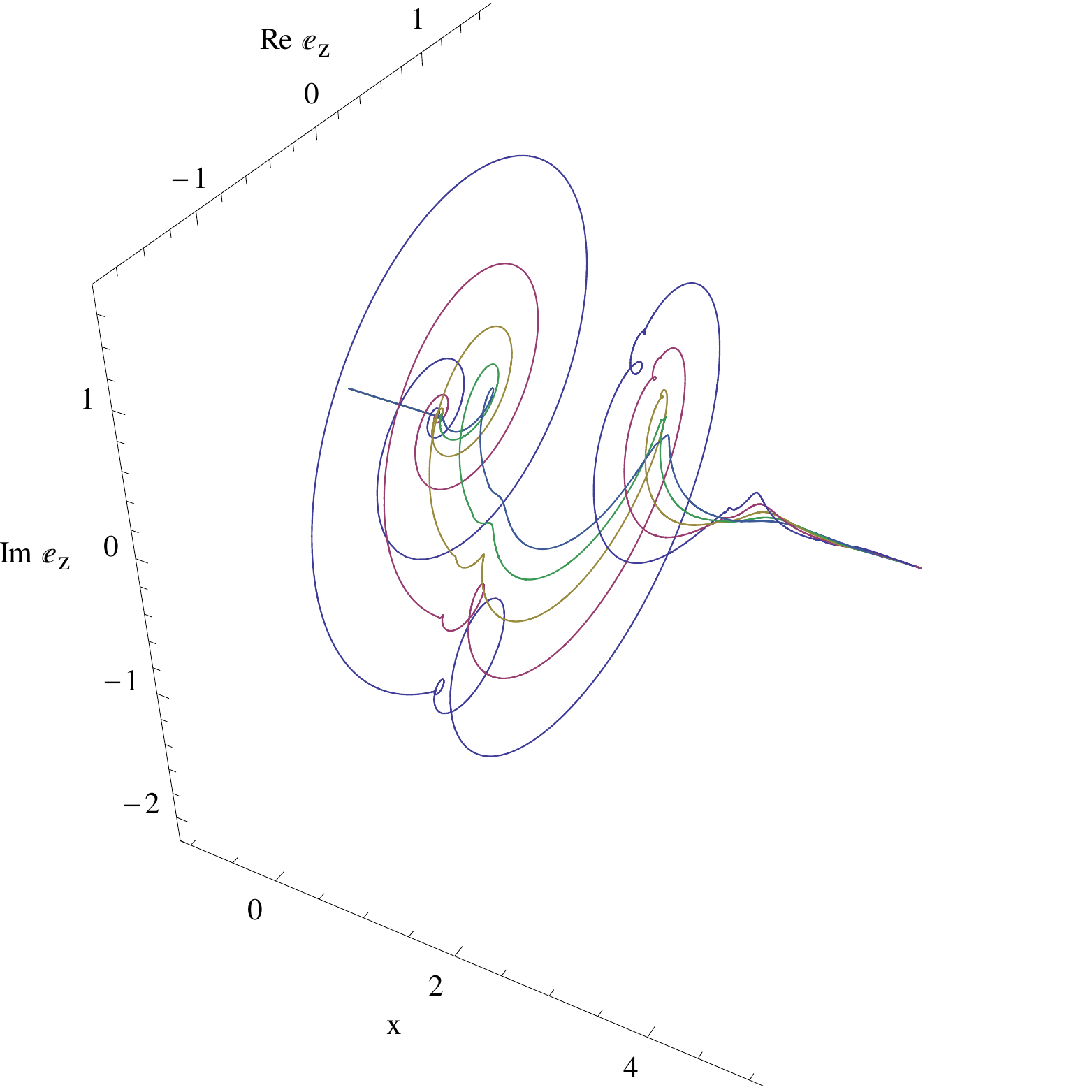}
\caption{$E_{z,a}$ for $z = (3 + \frac{k}{4}) + 3i$, $k = 0, 1, \ldots, 4$ and $a=1.7$. Top right: $\re E_{z,a}$, top left: $\im E_{z,a}$, bottom: Three-dimensional rendering of $E_{z,a}$.}\label{fig4}
\end{center}
\end{figure}
\begin{remark}
Complex polynomial and and exponential B-splines or order $z\in \C_{>1}$ are two-parameter families of functions assigning to each point $x\in [0,\infty)$ both a real value and a single direction given by $\im z$. For several applications however, such as geophysical data processing or multichannel data, more than one independent direction is required. For this purpose, the complex order is replaced by a quaternionic or more generally a hypercomplex order. We refer the interested reader to \cite{Hm17,hm18,m19} for these extensions in the case of polynomial B-splines.
\end{remark}
\section{Self-Referential Functions}\label{selfref}

In this section, we briefly review the concept of self-referential function. For more details and proofs we refer the interested reader to, for instance, \cite{barn,bhm,h,m10,m16,nav} or any other of the numerous publications in fractal interpolation theory. 

In the following, the symbol $\N_N :=\{1, \ldots, N\}$ denotes the initial segment of length $N$ of $\N$. Further, we assume that $N\geq 2$.

Let $I$ be a nonempty interval in $\R$ and suppose that $\{L_n : I\to I : n \in \N_N\}$ is a family of bijections with the property that $\{L_n(I) : n \in\N_N\}$ forms a partition of $I$, i.e., 
\begin{equation}\label{p} 
I = \bigcup_{n=1}^N L_n (I), \quad\text{and}\quad  L_n(I)\cap L_m(I) = \emptyset,\quad\forall n\neq m\in \N_N. 
\end{equation}

Denote by $B(I) := B(I, \R)$ the set
\[
B(I) := \{f : I\to \R : f \, \textrm{bounded}\}.
\]
$(B(I), d)$ becomes a complete metric space when endowed with the metric
\[
d(f,g): = {\sup_{x\in I}} \,|f(x) - g(x)|,
\]
where $|\cdot|$ denotes the Euclidean norm on $\R$.

Let $f, b\in B(I)$ be arbitrary. Consider the Read-Bajraktarev{\'i}c (RB) operator $T: B(I)\to B(I)$ defined on each subinterval $L_n(I)$ by
\be\label{T} 
Tg = f + \alpha_n\cdot (g-b)\circ L_n^{-1},\quad n\in \N_N,
\ee 
with $\alpha_n\in \R$. Under the assumption that $\alpha:= \max\{\alpha_n : n\in \N_N\} < 1$, it follows from the Banach fixed point theorem that $T$ has a unique fixed point $f^*\in B(I)$. This fixed point satisfies the \emph{self-referential equation}
\be\label{eq6.3}
f^* = f + \alpha_n \cdot (f^*-b)\circ L_n^{-1}, \quad \textrm{on} \quad L_n(I),\,\, n\in \N_N.
\ee
Any function in $B(I)$ which satisfies an equation of the form \eqref{eq6.3} is termed a \emph{self-referential function of type $B(I)$}. The functions $f$ and $b$ are called \emph{seed function}, respectively, \emph{base function}.

Note that $f^*$ can be iteratively obtained as the limit of the sequence $\{g_k\}$ defined by
\begin{gather}
g_k := T g_{k-1} = f + \alpha_n \cdot (g_{k-1} - b)\circ L_n^{-1}, \quad \text{on $L_n(I)$}, \quad k\in \N,
\end{gather}
for an arbitrary $g_0\in B(I)$.
\begin{remark}
The fixed point $f^*$ of an RB operator has the property that $\gr f^*$ is made up of a finite number of copies of itself and is therefore, in general, a fractal set. For this reason, $f^*$ is also called a \emph{fractal function} \cite{barn,h,m16}. 
\end{remark}
\begin{remark}
Self-referential functions defined on function spaces other than $B(I)$ can be constructed as well. Examples include, among others,  the smoothness spaces $C^r (I)$, the Lebesgue spaces $L^p (I)$, and the Besov spaces $B^s_{p,q} (I)$. (Cf., for instance, \cite{bhm,m97,m14,m18,m16}.) To ensure that the RB operator $T$ maps a function space into itself, additional conditions at the points $\{L_n (\partial I)\}$, $n\in\N_{N}$, may have to be imposed.
\end{remark}
\begin{remark}
For a given finite set of bijections $\{L_n\}$ or, equivalently, a given partition of $I$ yielding a finite set of bijections, the fixed point $f^*$ depends on the functions $f$ and $b$ as well as the vertical scaling factors $\{\alpha_n\}$. The interested reader may want to consult \cite{nm} in the former case.
\end{remark}
\begin{remark}
For a varying $N$-tuple $\balpha := (\alpha_1, \ldots, \alpha_N)$, the fixed point $f^*$ actually defines an uncountable family $f^{\balpha}$ of self-referential functions indexed by $\balpha\in (-1,1)^N$. Such sets of self-referential functions were termed $\alpha$-fractal functions and considered as the image of an operator $\mathcal{F}^\balpha$, $f\mapsto f^\balpha$. (Cf., i.e., \cite{nav}.)
\end{remark}

\section{Self-Referential Polynomial and Exponential B-Splines}\label{selfrefbexp}

In this section, we consider some fractal extensions of the classical as well as the complex polynomial and exponential B-splines. 

\subsection{Polynomial and exponential splines of integral order}
For this purpose, let $B_N$ be the cardinal polynomial B-spline of order $N\geq 2$ as in \eqref{eq2.3a}. Let $I:= \supp B_N = [0,N]$ and define bijections $L_n :I\to I$, $n\in \N_N$, by
\[
L_n (I) := \begin{cases} [n-1, n), & n \in \N_{N-1};\\ [N-1,N], & n = N.
\end{cases}
\]
Now choose $f:= B_N$ and $b\equiv 0$. Suppose $\alpha:= \max\{\alpha_n : n\in \N_N\} < 1$. Then the RB operator $T$ reads
\[
T g = B_N + \alpha_n \cdot g\circ L_n^{-1},\quad \text{on $L_n(I)$}, \quad n\in \N_N,
\]
for any $g\in B(I)$. As $B_N\in C^{N-2}$, we additionally require $g\in C^{N-2}(I)$ and impose the joint-up conditions
\be\label{eq7.1}
\forall\,m\in\N_{N-1}\,:(Tg)^{(\nu)} (m-) = (Tg)^{(\nu)} (m+), \quad \nu = 0, 1, \ldots, N-2.
\ee
Conditions \eqref{eq7.1} guarantee that $Tg\in C^{N-2}(I)$ and as $C^{N-2}(I)$ becomes a Banach space under the norm $\sum\limits_{\nu=0}^{N-2} \|(\cdot)^{(\nu)}\|_\infty$, the unique fixed point $\fB_N$ of $T$ is an element of $C^{N-2}(I)$ and a self-referential function:
\be\label{eq7.2}
\fB_N = B_N + \alpha_n \cdot \fB_N \circ L_n^{-1}, \quad \text{on $L_n(I)$}, \quad n\in \N_N.
\ee
As the fixed point $\fB_N$ depends continuously on the set of parameters $\balpha := (\alpha_1, \ldots, \alpha_N)\in (-1,1)^N$, we also write $\fB_N(\balpha)$ should the need arise. Hence, \eqref{eq7.2} defines an uncountable family of functions parametrized by $\balpha$. Clearly, $\balpha = 0$ reproduces the seed function $B_N$. (See also \cite{ns}.)

Figure \ref{fig5} depicts two such fractal polynomial B-splines: the linear $\fB_2((\frac34,\frac34))$ and the quadratic $\fB_3 ((\frac14,\frac14,\frac14))$. Note that $\fB_3 ((\frac14,\frac14,\frac14))$ is differentiable on $[0,3]$ and its graph is made up of three copies of itself. 

\begin{figure}[h!]
\begin{center}
\includegraphics[width=4cm,height=2.5cm]{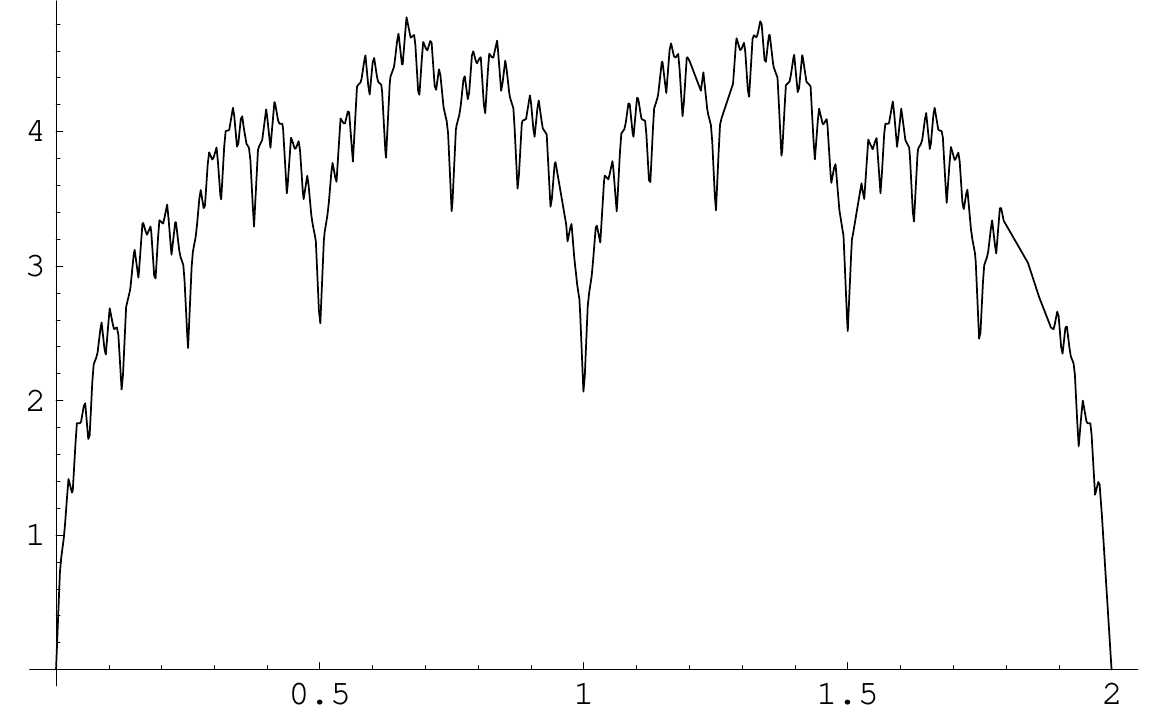}
\hspace*{1cm}\includegraphics[width=4cm,height=2.5cm]{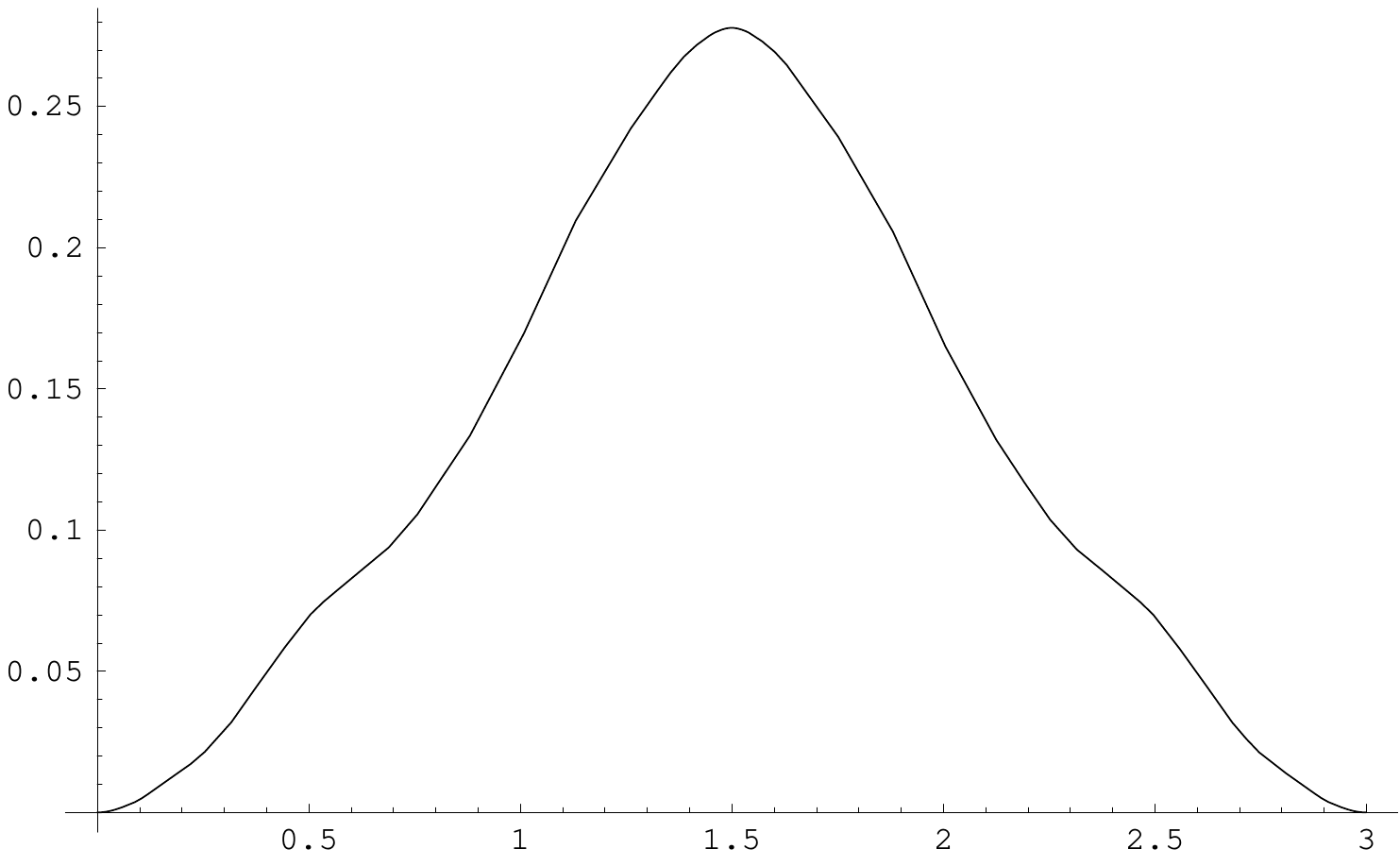}
\caption{A linear (left) and a quadratic (right) fractal polynomial B-spline.}\label{fig5}
\end{center}
\end{figure}

In a similar fashion, we can take $I:=[0,N]$, $f:= E_{N,\ba}$, and set again $b\equiv 0$ to generate an uncountable family of fractal analogues of the classical exponential B-splines $E_{N,\ba}$. The RB operator then reads
\[
T g = E_{N,\ba} + \alpha_n \cdot g\circ L_n^{-1},\quad \text{on $L_n(I)$}, \quad n\in \N_N,
\]
for any $g\in C(I)$ (as the functions $E_{N,\ba}$ are continuous on $I$). As above, we choose $\balpha\in (-1,1)^N$ and impose the continuity conditions
\[
Tg(m-) = Tg(m+), \quad m\in \N_{N-1}.
\]
Under these conditions, $T$ is well-defined and contractive from $C(I)$ into itself. Its unique fixed point $\fE_{N,\ba} := \fE_{N,\ba}(\balpha)$ satisfies the self-referential equation
\be\label{eq7.3}
\fE_{N,\ba} = E_{N,\ba}+ \alpha_n \cdot \fE_{N,\ba} \circ L_n^{-1}, \quad \text{on $L_n(I)$}, \quad n\in \N_N.
\ee
\noindent
In Figures \ref{fig6} and \ref{fig7}, two fractal exponential B-splines are depicted.
\begin{figure}[h!]
\begin{center}
\includegraphics[width=4cm,height=2cm]{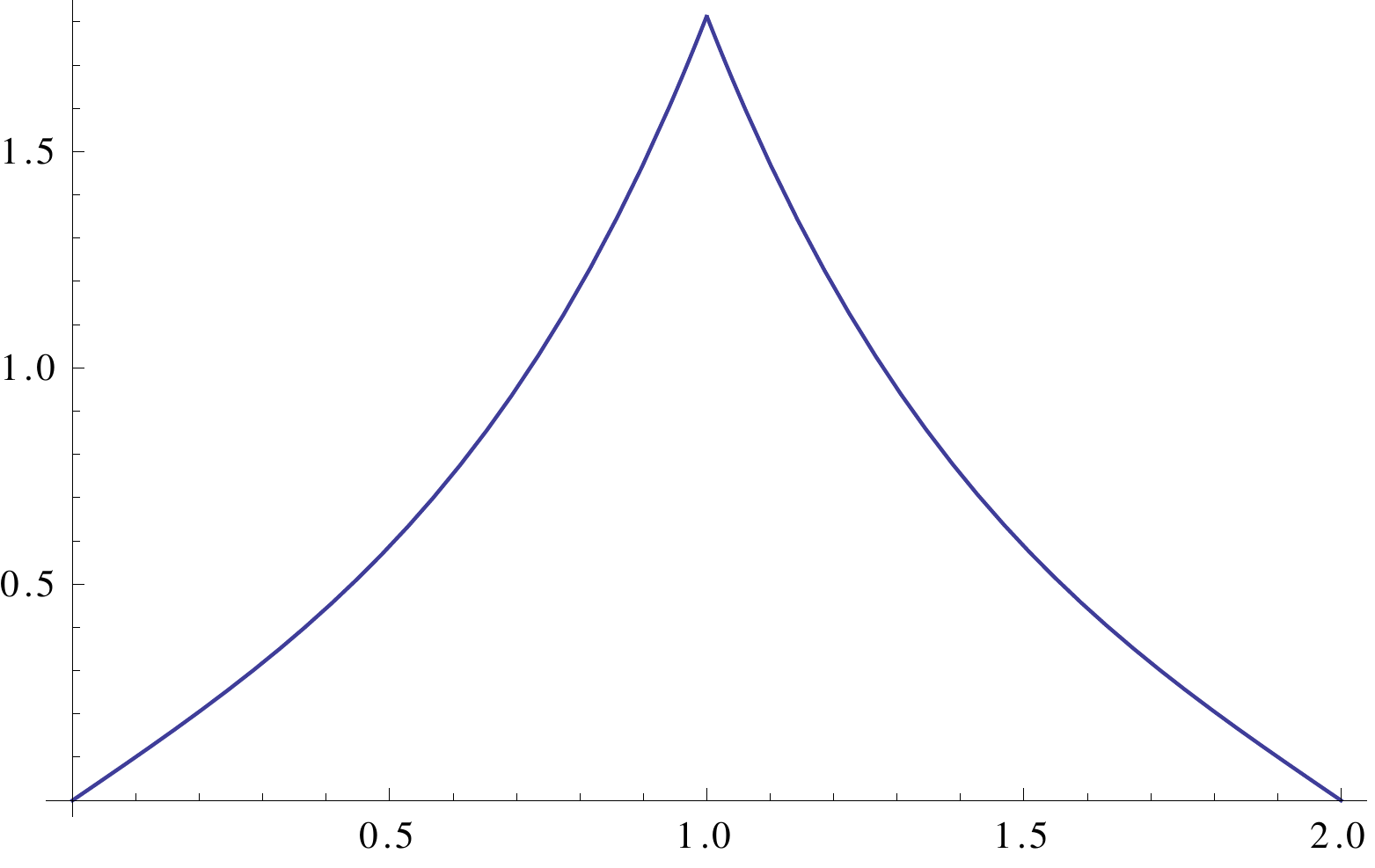}
\hspace*{1cm}\includegraphics[width=4cm,height=2cm]{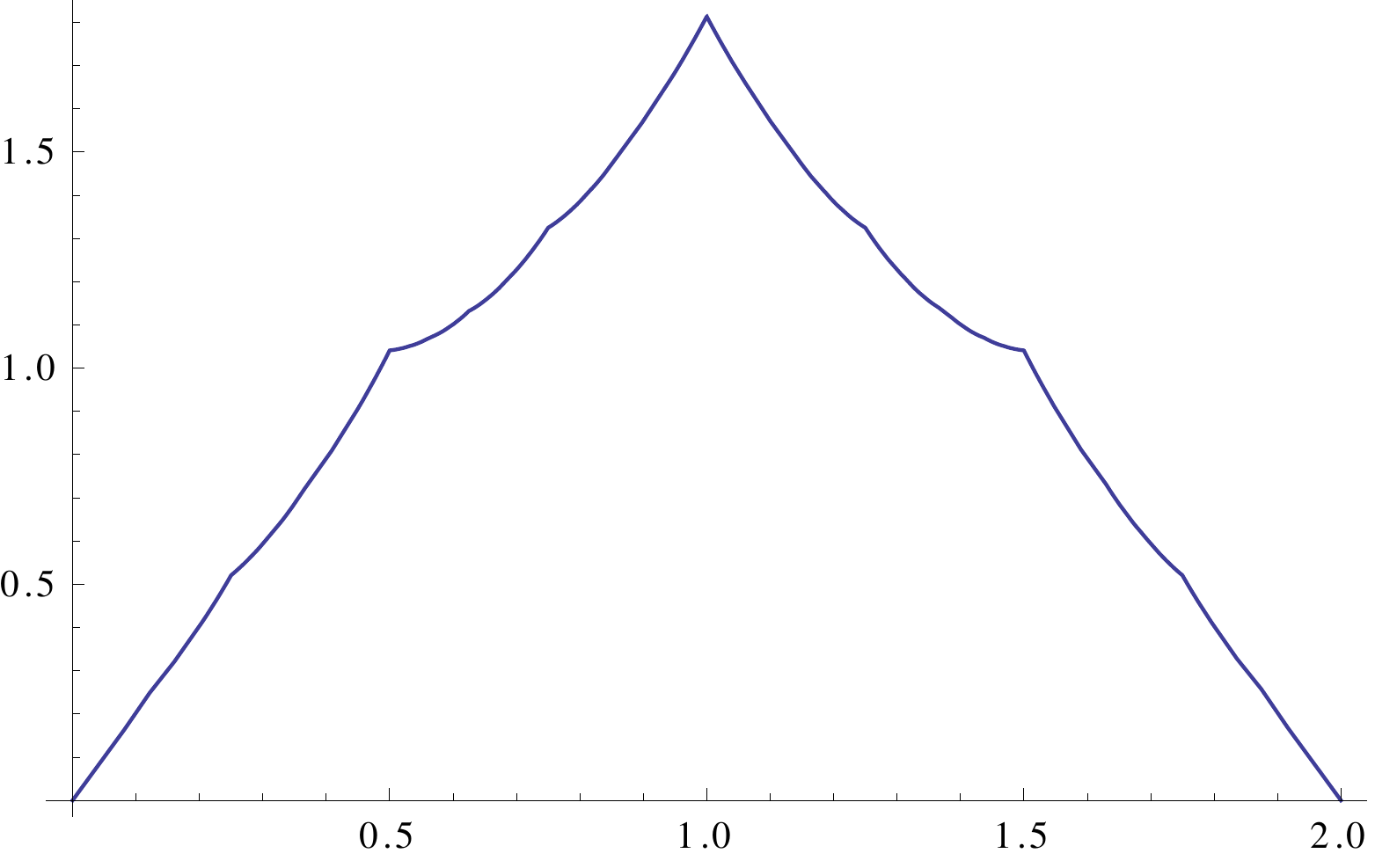}
\caption{The graphs of $E_{2, (2,-2)}$ and $\fE_{2, (2,-2)}(\frac14,\frac14)$}\label{fig6}
\end{center}
\end{figure}

\begin{figure}[h!]
\begin{center}
\includegraphics[width=4cm,height=2cm]{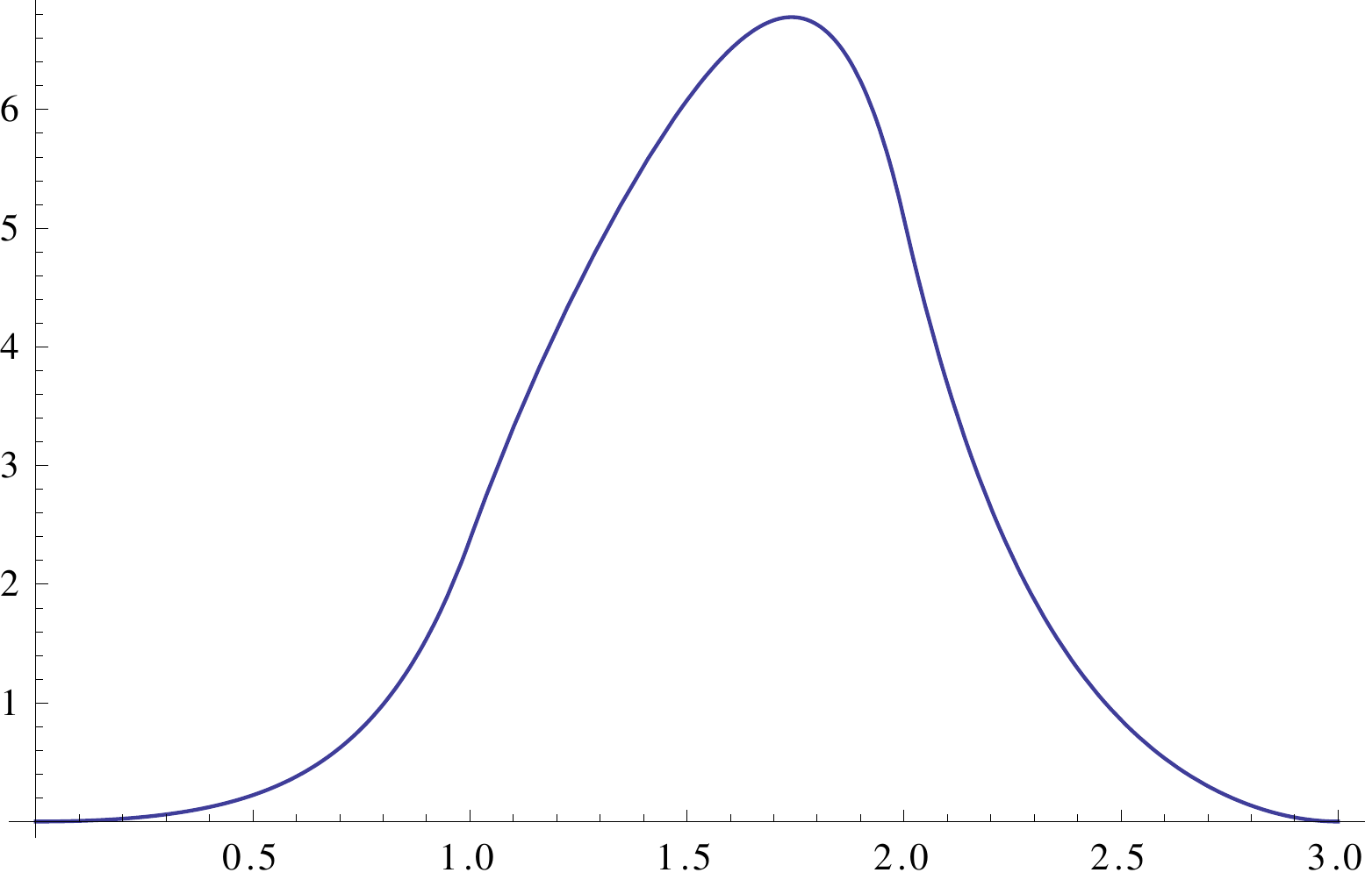}
\hspace*{1cm}\includegraphics[width=4cm,height=2cm]{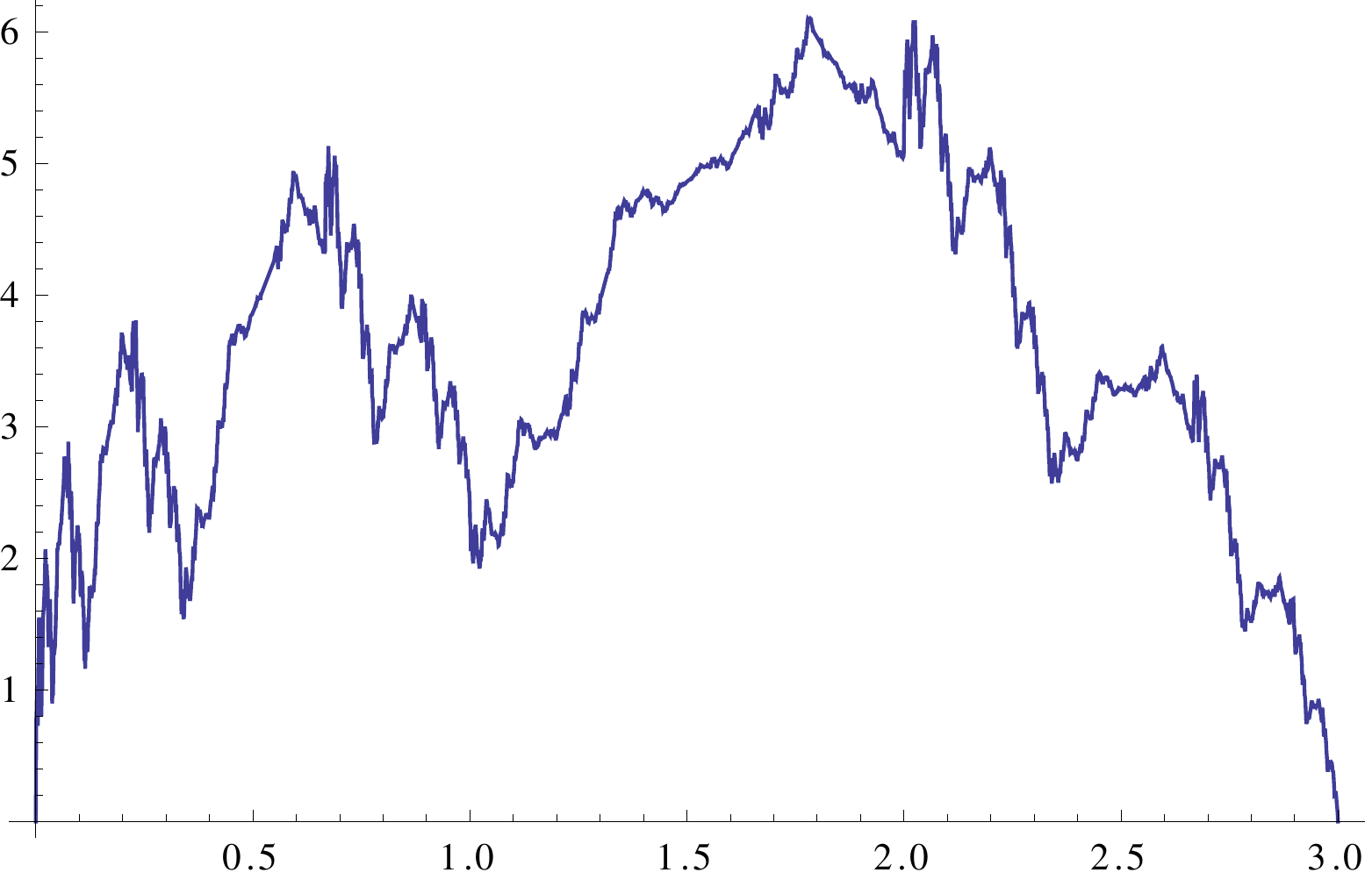}
\caption{The graphs of $E_{3, (4,-3,1)}$ and $\fE_{3, (4,-3,1)}(\frac34,-\frac14,\frac12)$}\label{fig7}
\end{center}
\end{figure}

\subsection{Polynomial and exponential B-splines of complex order}

In order to derive the fractal extensions of polynomial and exponential B-splines of complex order, we need to take into account the fact that the support of $B_z$ and $E_{z,a}$ is the unbounded interval $I:=[0,\infty)$ and extend the above construction to this setting. 

To be specific, suppose that the bijections $L_n$, $n\in \N_N$, are such that $L_n (I)$, $n\in \N_{N-1}$, is bounded on $\R$ and $L_N (I)$ unbounded. As before, we require that Eqn. \eqref{p} holds. We note that this set-up is an important special case of a general approach investigated in \cite{m18}.

To this end, we introduce the Banach space $(C_{0,0} (\R^+_0),\|\cdot\|_\infty)$ given by 
\[
C_{0,0} (\R^+_0) := C_{0,0} (\R^+_0, \R) :=\left\{f\in C(\R^+_0, \R) : f(0) = 0 \wedge\lim_{x\to \infty} f(x) = 0\right\}.
\]
As  $B_z$ and $E_{z,a}$ are continuous functions of the time variable $x$, vanish at $x = 0$, and satisfy $\lim\limits_{x\to\infty} B_z(x) = 0 = \lim\limits_{x\to\infty} E_{z,a}(x)$, we need to impose conditions on the RB operator $T$ in Eqn. \eqref{T} to map $C_{0,0} (\R^+_0)$ into itself.

These conditions read as follows. For $n \in \N_{N-1}$, denote 
\be
\begin{aligned}
L_n (0) &=: x_{n-1}, \label{eq6.4}\\
L_n (\infty) &=: x_n,
\end{aligned}
\ee
and for $n:=N$:
\be
\begin{aligned}
L_N (0) &=: x_{N-1} ,\\
L_n (\infty) &= \infty\label{eq6.7}.
\end{aligned}
\ee
Here, we used the shorthand notation $f(\infty) := \lim\limits_{x\to\infty} f(x)$ for a function $f$.

As a base function, we choose again $b\equiv 0$ on $[0,\infty)$ and require that, for $n\in \N_{N-1}$, 
\be\label{eq6.8}
T g(x_n-) = Tg (x_n+)
\ee
or, equivalently, 
\be\label{eq6.9}
T g(L_n (\infty)) = Tg (L_{n+1}(0)),
\ee 
with the obvious modification for $n=N$. 

\begin{theorem}\label{th3}
Suppose bijections $L_n : \R^+_0\to \R^+_0$ are chosen such that $\{L_n (\R^+_0)\}_{n\in \N_N}$ forms a partition of $[0,\infty)$ subject to \eqref{eq6.4} and \eqref{eq6.7}. Further suppose that $T: C_{0,0} (\R^+_0)\to C_{0,0} (\R^+_0)$ is given by
\be\label{eq6.10}
Tg = f + \alpha_n\cdot g\circ L_n^{-1},
\ee
and satisfies \eqref{eq6.8}, where $f\in C_{0,0} (\R^+_0)$ is arbitrary and $\alpha := \max\limits\{|\alpha_n| : n\in \N_N\} < 1$. Then $T$ is well-defined and contractive on $(C_{0,0} (\R^+_0),\|\cdot\|_\infty)$ with Lipschitz constant $\alpha$. 
\end{theorem}
\begin{proof}
The conditions on the bijections $\{L_n\}$ and the join-up conditions \eqref{eq6.8} guarantee that $T$ is well-defined and maps $C_{0,0} (\R^+_0)$ into itself. To establish that $T$ is contractive on $(C_{0,0} (\R^+_0),\|\cdot\|_\infty)$ with Lipschitz constant $\alpha$ is straightforward.
\end{proof}
The unique fixed point $f^*\in C_{0,0} (\R^+_0)$ of $T$ as defined in Eqn. \eqref{eq6.10} is called a \emph{self-referential function of class $C_{0,0} (\R^+_0)$}.

\begin{remark}\label{rem5}
Note that Theorem \ref{th3} also holds for the Banach spaces $(C_{0,0} (\R_0^+, \C), \|\cdot \|_\infty)$.
\end{remark}

As noted above, for varying $\balpha = (\alpha_1, \ldots, \alpha_N)$ subject to $\alpha := \max\limits\{|\alpha_n| : n\in \N_N\} < 1$, $f^*$ actually defines an uncountably infinite family $f^\balpha$ of self-referential functions containing the seed function $f$.

As two prominent examples of how to obtain the fractal extension of functions in $C_{0,0} (\R^+_0)$, we consider $f = B_z$ and $f = E_{z,a}$. For this purpose and the sake of presentation, we choose $N:=2$ and define bijections $L_n: \R^+_0\to \R^+_0$ by
\[
L_1(x) := {2}\,{\pi^{-1}}\arctan x\quad\text{and}\quad L_2(x) := x+1.
\]
Then $[0,\infty) = L_1 ([0,\infty)) \cup L_2 ([0,\infty)) = [0,1) \cup [1,\infty)$.

Now select $f := B_z$, respectively, $f = E_{z,a}$, choose $\alpha_n \in (-1,1)$, $n = 1,2$, and define RB operators
\[
T_1 g := B_z + \alpha_1\, g\circ \tan\, (\tfrac{\pi x}{2})\big\vert_{[0,1)} + \alpha_2\, g(x-1)\big\vert_{[1,\infty)}.
\]
and 
\[
T_2 g := E_{z,a} + \alpha_1\, g\circ \tan\, (\tfrac{\pi x}{2})\big\vert_{[0,1)} + \alpha_2\, g(x-1)\big\vert_{[1,\infty)}.
\]
By Theorem \ref{th3} and Remark \ref{rem5}, we obtain the fractal extensions of $B_z$ and $E_{z,a}$ as the fixed points $\fB_z (\balpha)$, respectively, $\fE_{z,a}(\balpha)$ of the RB operators $T_1$ and $T_2$:
\[
\fB_z = B_z + \alpha_1\, \fB_z\circ \tan\, (\tfrac{\pi x}{2})\big\vert_{[0,1)} + \alpha_2\, \fB_z(x-1)\big\vert_{[1,\infty)}.
\]
and 
\[
\fE_{z,a} = E_z^a + \alpha_1\, \fE_{z,a}\circ \tan\, (\tfrac{\pi x}{2})\big\vert_{[0,1)} + \alpha_2\, \fE_{z,a}(x-1)\big\vert_{[1,\infty)}.
\]
In Figures \ref{fig8} and \ref{fig9}, the graphs of $\fB_{\pi+i} (\frac34,-\frac12)$ and $\fE_{\sqrt{2}+i,1}(\frac34,-\frac12)$ are displayed.

\begin{figure}[h!]
\begin{center}
\includegraphics[width = 5cm, height = 3cm]{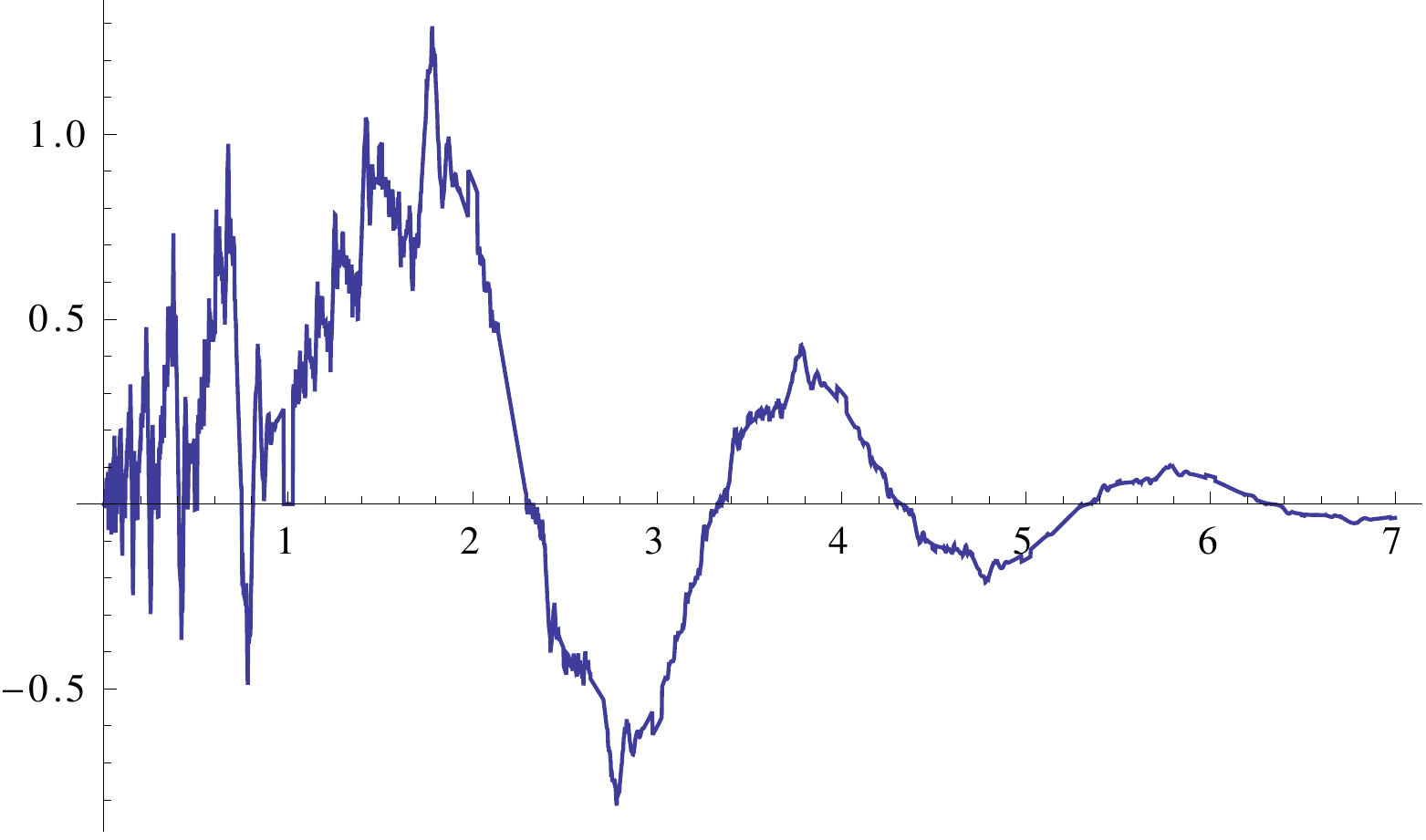}\hspace{1.5cm}
\includegraphics[width = 5cm, height = 3cm]{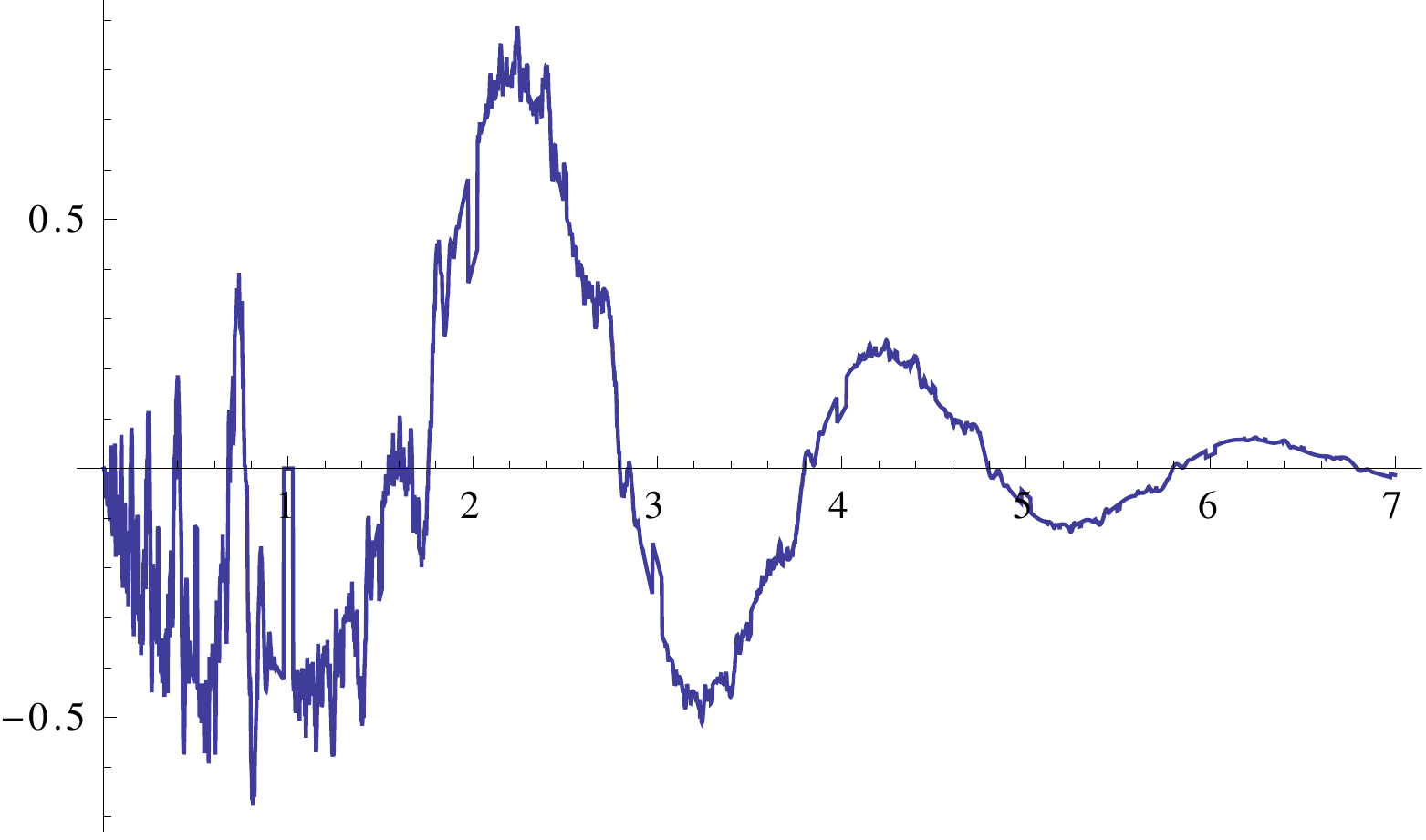}\\
\includegraphics[width = 6cm, height = 4cm]{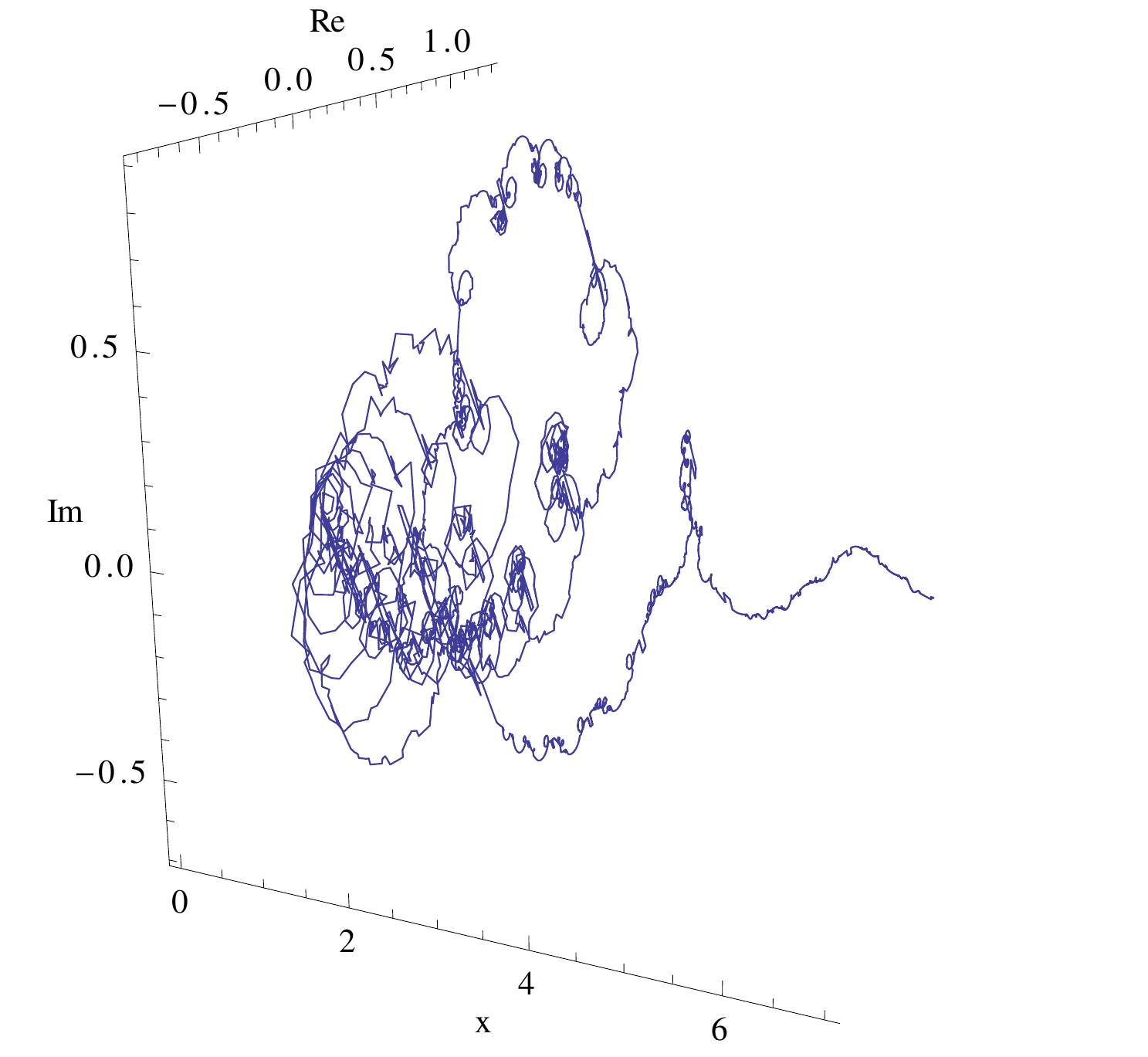}
\caption{Top right: $\re \fB_{\pi+i} (\frac34,-\frac12)$, top left: $\im \fB_{\pi+i} (\frac34,-\frac12)$, bottom: Three-dimensional rendering of $\fB_{\pi+i} (\frac34,-\frac12)$.}\label{fig8}
\end{center}
\end{figure}

\begin{figure}[h!]
\begin{center}
\includegraphics[width = 5cm, height = 3cm]{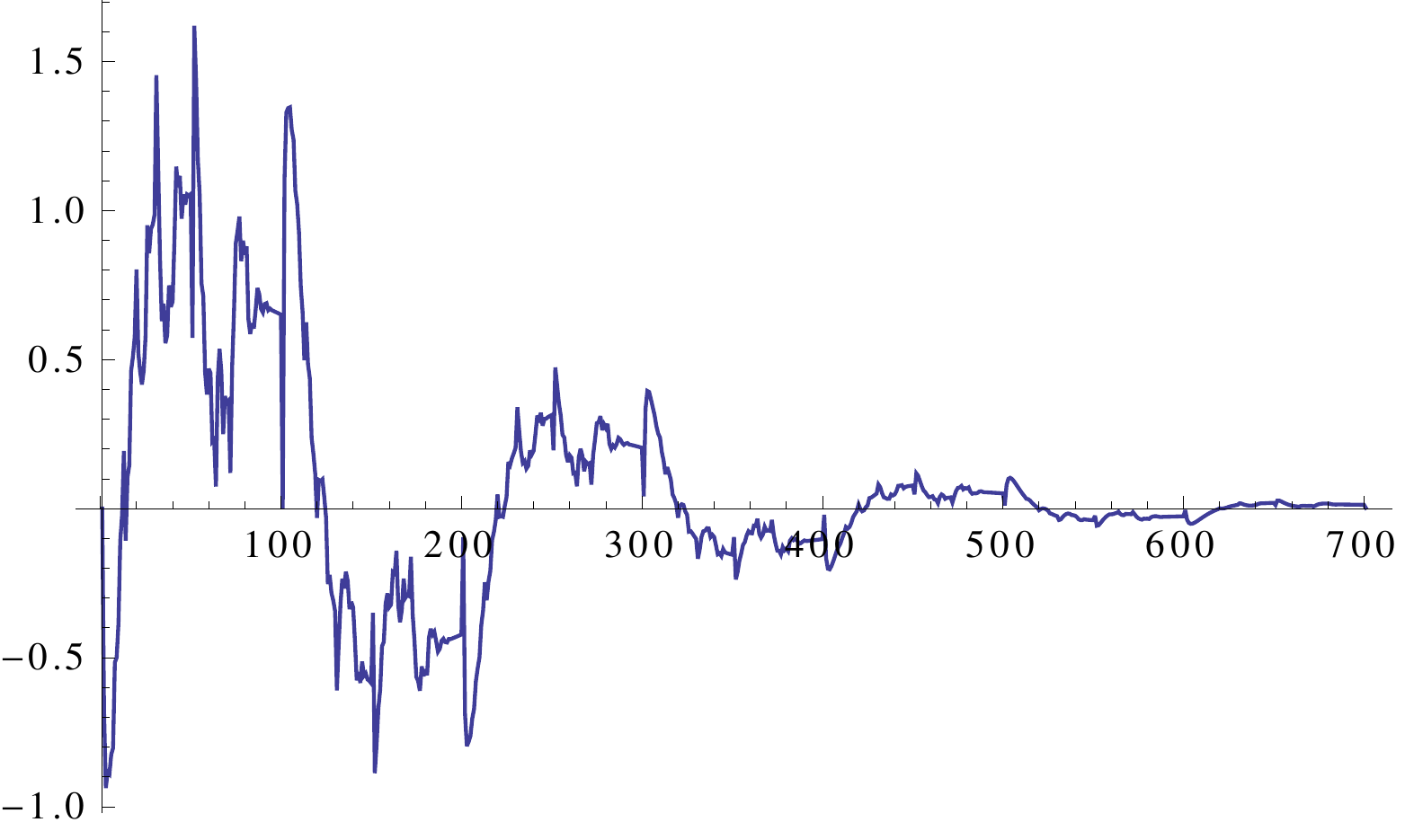}\hspace{1.5cm}
\includegraphics[width = 5cm, height = 3cm]{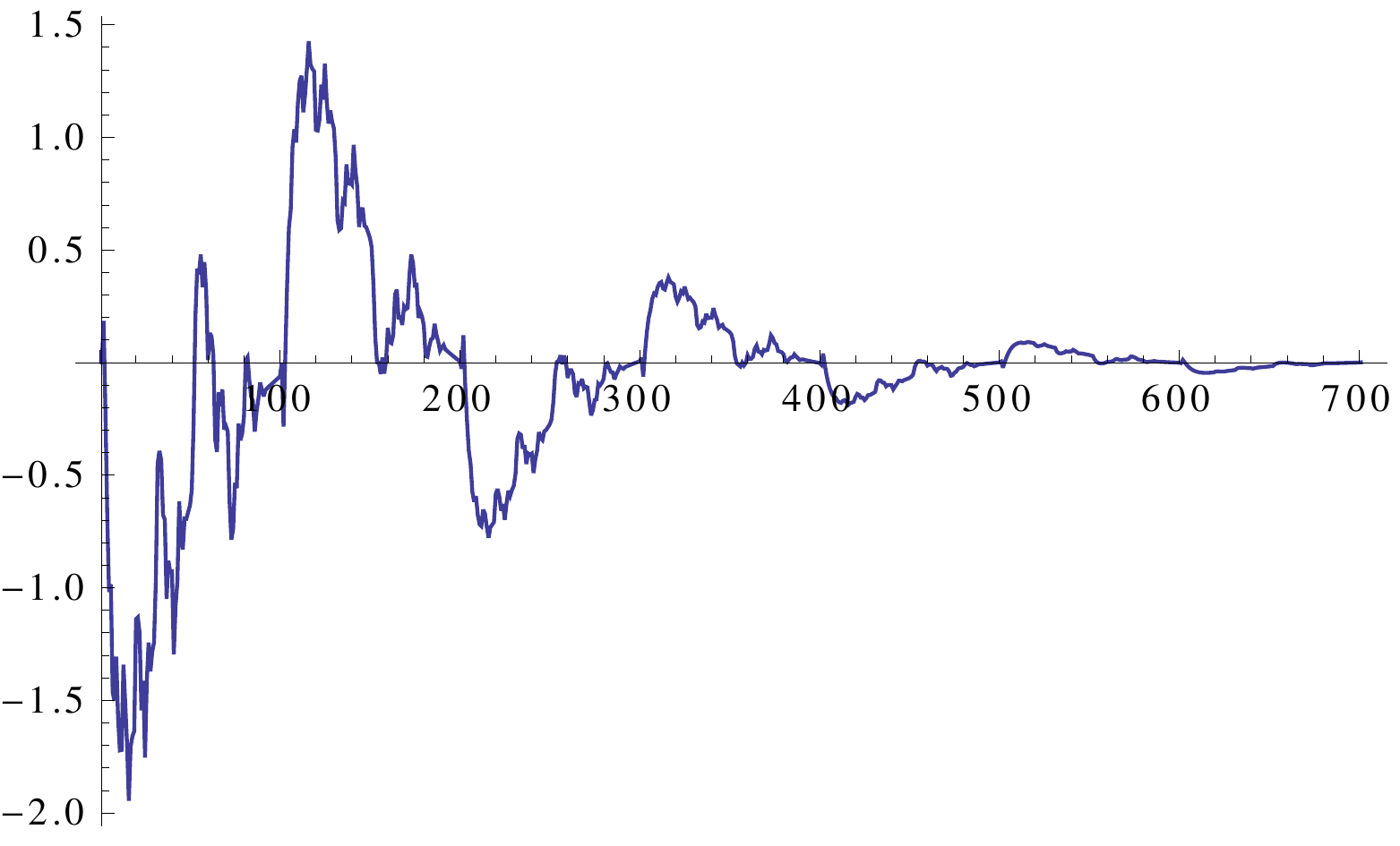}\\
\includegraphics[width = 6cm, height = 4cm]{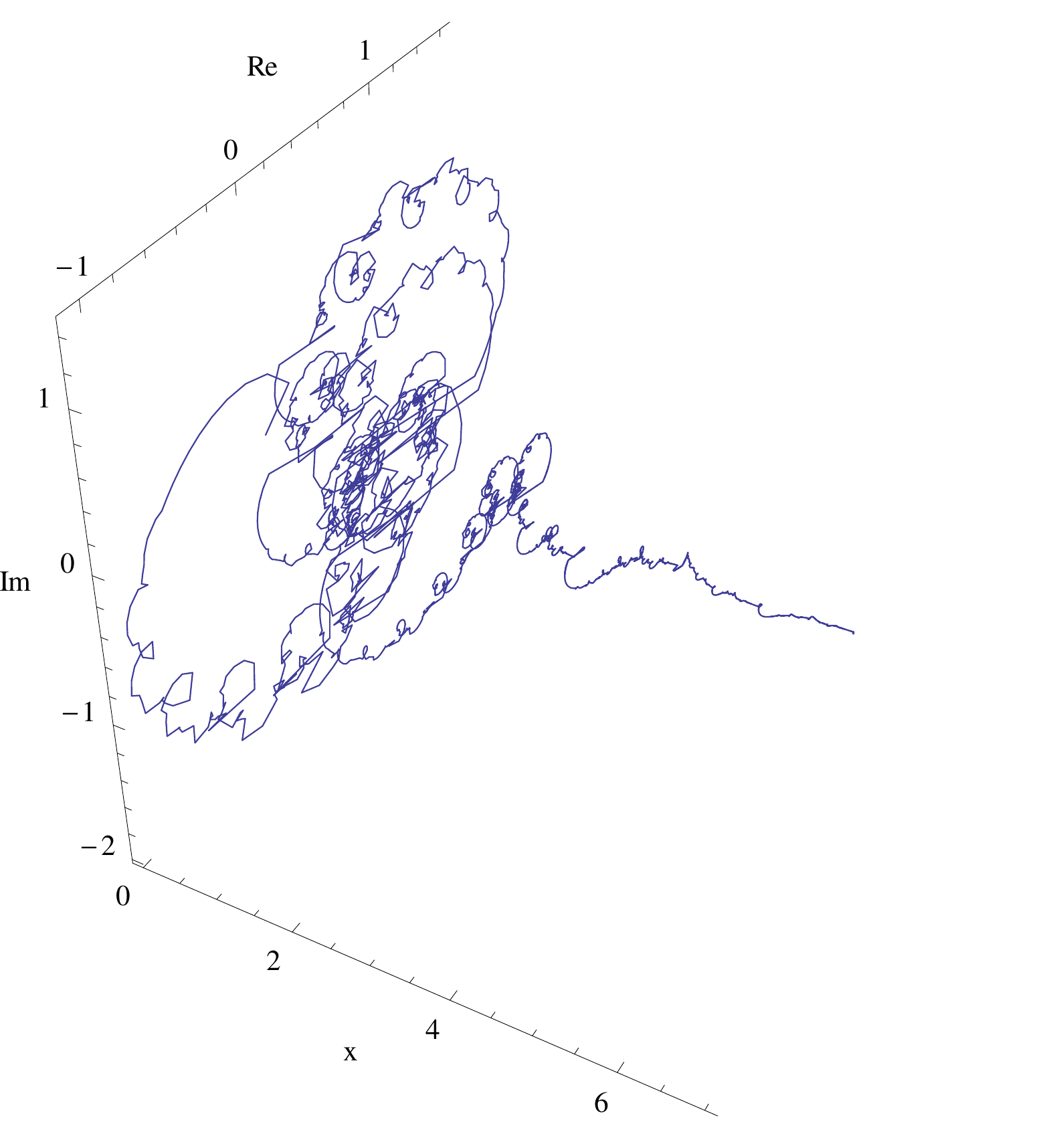}
\caption{Top right: $\re \fE_{\sqrt{2}+i,1}(\frac34,-\frac12)$, top left: $\im \fE_{\sqrt{2}+i,1}(\frac34,-\frac12)$, bottom: Three-dimensional rendering of $\fE_{\sqrt{2}+i,1}(\frac34,-\frac12)$. The length unit on the $x$-axis for the graphs on top is $\frac{1}{100}$.}\label{fig9}
\end{center}
\end{figure}

\begin{remark}
The families of self-referential functions supported on the interval $I=[0, \infty)$ not only depend on $\balpha$ but also on the partition induced by the bijections $L_n$ on $I$. Denoting the collection of all such partitions by $\Pi = \Pi_N$, the set of fixed points $f^{\balpha}$ should more precisely be written as $f^\balpha_{\Pi}$ and regarded as a function $(-1,1)^N\times \Pi\to f^*$.
\end{remark}
\bibliographystyle{plain}
\bibliography{Jaca}

\begin{thebibliography}{10}

\bibitem{ammar}
G.~Ammar, W.~Dayawansa, and C.~Martin.
\newblock Exponential interpolation theory: {Theory} and numerical algorithms.
\newblock {\em Appl. Math. Comput.}, 41:189--232, 1991.

\bibitem{barn}
M.~F. Barnsley.
\newblock Fractal functions and interpolation.
\newblock {\em Const. Approx.}, 2:303--329, 1986.

\bibitem{bhm}
M.~F. Barnsley, M.~Hegland, and P.~Massopust.
\newblock Numerics and fractals.
\newblock {\em Bull. Inst. Math. Acad. Sinica (N.S.)}, 9(3):389--430, 2014.

\bibitem{cm}
O.~Christensen and P.~Massopust.
\newblock Exponential {B}-splines and the parition of unity property.
\newblock {\em Adv. Comput. Math.}, 37:301--318, 2012.

\bibitem{dm1}
W.~Dahmen and C.~A. Micchelli.
\newblock On the theory and applications of exponential splines.
\newblock In C.~K. Chui, L.~L. Schumaker, and F.~I. Utreras, editors, {\em
  Topics in Multivariate Approximation}. Academic Press, Boston, 1987.

\bibitem{deboor}
Carl de~Boor.
\newblock {\em A {P}ractical {G}uide to {S}plines}.
\newblock Number~27 in Applied Mathematical Sciences. Springer Verlag, 2001.

\bibitem{fub}
B.~Forster, M.~Unser, and T.~Blu.
\newblock Complex {B}-splines.
\newblock {\em Appl. Comput. Harm. Anal.}, 20:261--282, 2006.

\bibitem{Hm17}
J.~Hogan and P.~Massopust.
\newblock Quaternionic {B}-{S}plines.
\newblock {\em J. Approx. Th.}, 224:43--65, 2017.

\bibitem{hm18}
J.~Hogan and P.~Massopust.
\newblock Quaternionic fundamental cardinal splines: Interpolation and
  sampling.
\newblock {\em arxiv.org/abs/1804.06638}, pages 1--24, 2018.

\bibitem{h}
J.~E. Hutchinson.
\newblock Fractals and self-similarity.
\newblock {\em Indiana J. Math.}, 30(5):713--747, 1981.

\bibitem{m97}
P.~Massopust.
\newblock Fractal functions and their applications.
\newblock {\em Chaos, Solitons, \& Fractals}, 8(2):171--190, 1997.

\bibitem{m05}
P.~Massopust.
\newblock Fractal functions, splines, and besov and triebel-lizorkin spaces.
\newblock In J.~L\'evy-V\'ehel and E.~Lutton, editors, {\em Fractals in
  Engineering: New trends and applications}, pages 21--32. Springer Verlag,
  London, 2005.

\bibitem{m10}
P.~Massopust.
\newblock {\em Interpolation and Approximation with Splines and Fractals}.
\newblock Oxford University Press, 2010.

\bibitem{m}
P.~Massopust.
\newblock Exponential splines of complex order.
\newblock {\em Contemp. {M}ath.}, 626:87--105, 2014.

\bibitem{m14}
P.~Massopust.
\newblock Local fractal functions and function spaces.
\newblock In C.~Bandt et~al., editor, {\em Fractals, {W}avelets, and their
  {A}pplications}, Springer Proceedings in Mathematics \& Statistics, pages
  245--270. Springer Verlag, 2014.

\bibitem{m18}
P.~Massopust.
\newblock Local fractal interpolation on unbounded domains.
\newblock {\em Proc. Edinburgh Math. Soc.}, 61:151--167, 2018.

\bibitem{m19}
P.~Massopust.
\newblock Splines and fractional differential operators.
\newblock {\em arxiv.org/abs/1901.11304}, pages 1--17, 2019.

\bibitem{m16}
Peter~R. Massopust.
\newblock {\em Fractal Functions, Fractal Surfaces, and Wavelets}.
\newblock Academic Press, New York, 2nd edition, 2016.

\bibitem{mccartin}
B.~J. McCartin.
\newblock Theory of exponential splines.
\newblock {\em J. Approx. Th.}, 66:1--23, 1991.

\bibitem{nav}
M.~A. Navacsu\'{e}s.
\newblock Fractal polynomial interpolation.
\newblock {\em Z. Anal. Anwendungen}, 24(2):401--418, 2005.

\bibitem{ns}
M.~A. Navacsu\'{e}s and M.~V. Sebasti\'{a}n.
\newblock Fractal splines.
\newblock {\em Monograf\'{i}as del Seminario Matem\'{a}tico Garc\'{i}a de
  Galdeano}, 33(161--168), 2006.

\bibitem{nm}
M.~A. Navascu\'{e}s and P.~Massopust.
\newblock Fractal convolution - a new operation between functions.
\newblock {\em arxiv:1805.11316v1}, pages 1--21, 2018.

\bibitem{ol}
A.~Oppenheim and J.~Lim.
\newblock The importance of phase in signals.
\newblock {\em IEEE}, 69(5):529--541, 1981.

\bibitem{sakai1}
M.~Sakai and R.~A. Usmani.
\newblock On exponential {B}-splines.
\newblock {\em J. Approx. Th.}, 47:122--131, 1986.

\bibitem{schoen}
I.~J. Schoenberg.
\newblock Contributions to the problem of approximation of equidistant data by
  analytic functions.
\newblock {\em Quart. Appl. Math.}, 4:45--99, 112--141, 1946.

\bibitem{spaeth}
H.~Spaeth.
\newblock Exponential spline interpolation.
\newblock {\em Computing}, 4:225--233, 1969.

\bibitem{sb}
J.~Stoer and R.~Bulirsch.
\newblock {\em Introduction to Numerical Analysis}.
\newblock Springer Verlag, New York, 2nd edition, 1993.

\bibitem{ub}
M.~Unser and T.~Blu.
\newblock Fractional {S}plines and {Wavelets}.
\newblock {\em {SIAM} {R}eview}, 42(1):43--67, 2000.

\bibitem{unserblu05}
M.~Unser and T.~Blu.
\newblock Cardinal exponential splines: {P}art {I} -- theory and filtering
  algorithms.
\newblock {\em {IEEE} Trans. Signal Processing}, 53(4):1425--1438, 2005.

\bibitem{zoppou}
C.~Zoppou, S.~Roberts, and R.~J. Renka.
\newblock Exponential spline interpolation in characteristic based scheme for
  solving the advective--diffusion equation.
\newblock {\em Int. J. Numer. Meth. Fluids}, 33:429--452, 2000.

\end{thebibliography}
   
\end{document}